\documentclass[12pt, reqno]{amsart}
\pdfoutput=1
\usepackage{amsmath, amstext, amsbsy, amssymb, amscd}
\usepackage{amsxtra}
\usepackage{stmaryrd}
\usepackage{amscd}
\usepackage{amsthm}
\usepackage{amsfonts}
\usepackage{mathrsfs}
\usepackage{eucal}
\usepackage{color}
\usepackage[all]{xy}
\usepackage{tikz-cd}
\usepackage[CJKbookmarks=true]{hyperref}

\newtheorem{theorem}{Theorem}[section]
\newtheorem{lemma}[theorem]{Lemma}
\newtheorem{prop}[theorem]{Proposition}
\newtheorem{corollary}[theorem]{Corollary}
\theoremstyle{definition}

\newtheorem{example}[theorem]{Example}
\newtheorem{remark}[theorem]{Remark}

\numberwithin{equation}{section}

\def\Aut{\mathrm{Aut}}

\def\Der{\mathrm{Der}}

\def\Hom{\mathrm{Hom}}

\def\Ker{\mathrm{Ker}}

\def\Lie{\mathrm{Lie}}
\def\Mor{\mathrm{Mor}}

\def\Spec{\mathrm{Spec}}

\def\Ext{\mathrm{Ext}}

\def\gr{\mathrm{gr}}

\def\bbz{\bar{\bar{0}}}

\makeatletter
\newcommand{\rmnum}[1]{\romannumeral #1}
\newcommand{\Rmnum}[1]{\expandafter\@slowromancap\romannumeral #1@}
\makeatother

{\vskip-\lastskip\medskip
  \noindent
  {\em #1.}\enspace
  }%
{\qed\par\medskip
  }

\begin{document}
\title[On deformation quantizations of symplectic supervarieties]
{On deformation quantizations of symplectic supervarieties}
\author{Husileng Xiao}
\address{ School of mathematical Science, Harbin Engineering University, Harbin, 150001, China }\email{hslxiao@hrbeu.edu.cn}
\maketitle

\begin{abstract}
We classify deformation quantizations
of the symplectic supervarieties that are smooth and admissible. 
This generalizes the corresponding result of  Bezrukavnikov and Kaledin to the super case.
We relate the equivalence classes of quantizations of  supervarieties with that of their even reduced symplectic varieties. Finally, we prove that certain  nilpotent orbits of basic Lie superalgebras are admissible and split, and classify their deformation quantizations.
\end{abstract}

\section{Introduction}

The deformation quantization of Poisson manifold is an extensively studied topic in mathematics and physics. A fundamental problem is the existence and classification of the quantizations of a given Poisson manifold.  It was  solved systematically by Kontsevich based on his famous formality theorem \cite{Ko1}. Beyond the setting of real manifolds, deformation quantization is also considered in other geometric settings, such as holomorphic or algebraic cases \cite{Ko3, NT, Ye1}.   Let $X$ be an admissible smooth algebraic variety and $\{,\}_{\omega}$ be the Poisson bracket arising from a symplectic form $\omega$ on $X$. 
  Denote by  $Q(X,\omega)$ the set of  isomorphism classes of deformation quantizations of $(X,\{,\}_{\omega})$. Applying techniques from  formal geometry,  Bezrukavnikov and Kaledin constructed a period map from $Q(X,\omega)$ to $H_{\mathrm{dR}}^{2}(X)[[\hbar]]$, which is injective \cite{BK}.

Deformation quantization of supergeometry is an important topic, not only in physics but also with applications in non-super geometry. For instance, Voronov derived  the Atiyah-Singer index theorem on real manifolds using deformation quantizations of a suitable  symplectic supermanifold \cite{Vo}. For the most recent progress on this topic, see \cite{Am1,Am2}. In this paper, we study the deformation quantization  of   supervarieties 
over an algebraically closed field $\mathbf{k}$ of characteristic $0$. Inspired by Losev's works on semisimple Lie algebras, our main motivation is to apply these results in the representation theory of Lie superalgebras.

From now on, let $(X,\omega)$ be an admissible smooth supervariety equipped with an even symplectic form $\omega$.
Our main result is to construct an injective period map from $Q(X, \omega)$ to 
a subset of $H_{\mathrm{dR}}^{2}(X)[[\hbar]]$ (see Theorem \ref{thm_peroid}).   We use an approach similar  to that in \cite{BK}. Most of the arguments there can be generalized to the super case. However, some preparatory work on super geometric concepts is required. We present some fundamental properties  of $\mathcal{D}$-modules, de Rham cohomology, jet bundles and 
Harish-Chandra torsors in the setting of super algebraic geometry.
 As in the case of real supermanifolds, Polishchuk showed that the de Rham cohomology of $X$ is isomorphic to that of its even reduced variety $X_{\bbz}$ \cite{Po}. This phenomenon will be reflected in the quantizations. Assume that $X$  and $X_{\bbz}$ are admissible. We construct an injective map  $ \iota_Q:  Q(X, \omega)  \rightarrow Q(X_{\bbz}, \omega_{\bbz}) $ applying the period map. 

Finally, we provide an important class of supervarieties that satisfies the conditions of our main result. 
Let $\mathfrak{g}=\mathfrak{g}_{\bar{0}}+\mathfrak{g}_{\bar{1}}$ be a basic Lie superalgebra over $\mathbf{k}$, $\mathbb{O}$ be the coadjoint orbit of a nilpotent element $\chi \in  \mathfrak{g}_{\bar{0}}^*= \mathfrak{g}_{\bar{0}}$ under the corresponding supergroup $G$  and $\omega_{\chi}$  be the  Kostant-Kirillov symplectic form  on $\mathbb{O}$.
In the last section, we prove that a class of  super nilpotent orbits $\mathbb{O}$  are admissible and split.  This leads to a classification of   deformation quantizations of  $(\mathbb{O},\omega_{\chi})$.
In the pure even case $\mathfrak{g}=\mathfrak{g}_{\bar{0}}$, Losev  classified the filtered quantizations of  $\mathbb{O}$  based on the classification of deformation quantizations. Then he found some deep connections  among the finite-dimensional representations of the W-algebra $\mathrm{U}(\mathfrak{g},e)$, filtered quantizations of $\mathbb{O}$ and primitive ideals of the universal enveloping algebra $\mathrm{U}(\mathfrak{g})$. Furthermore, he formulated a version of  orbit method for  semisimple Lie algebras \cite{Lo2}.  The Harish-Chandra modules over quantizations of  $(\mathbb{O},\omega_{\chi})$ are recently considered in \cite{LY}. The main motivation of this paper is to develop a super version of these results. 

\subsection*{Conventions}
A supervariety $X$ will also be denoted by $(X_{\bbz}, \mathcal{O}_X)$ if we want to emphasize its 
reduced topological part $X_{\bbz}$ and the structure sheaf $\mathcal{O}_X$. For a sheaf $\mathcal{F}$ over $X$, the  sections over an open set $U \subset X_{\bbz} $ is denoted by 
$\mathcal{F}(U)$ or $\Gamma(U,\mathcal{F})$. For a vector space $V=V_{\bar{0}}+V_{\bar{1}}$, $V_{\hbar}$ denotes the  vector superspace $V[[\hbar]]$. For $ v \in V$, $|v| \in\{\bar{0},\bar{1}\}$ takes the degree of $v$,  is meaningful if and only if
$v$ is homogeneous. By a vector bundle on $X$, we mean a locally free sheaf of $\mathcal{O}_X$-modules with finite or infinite rank. The super symmetric algebra $\mathbf{S}[V_{\bar{0}}]\otimes \bigwedge V_{\bar{1}}$ of $V$ is denoted by $\mathbf{S}(V)$. For a pair of superschemes $X$ and $Y$, $\mathrm{Mor}(X,Y)$ denotes the set of morphims from $X$ to $Y$.

\section*{Acknowledgment}
The author thanks the anonymous referee for suggesting Reference \cite{Po} and for their numerous other helpful comments, 
which have significantly improved the quality of this paper. The referee pointed out several gaps in the proofs of Proposition \ref{prop_admi} and Theorem \ref{thm_5.4} in the previous version of the manuscript. This work is supported by the Natural Science Foundations of Heilongjiang Province, China(Grant No: YQ2023A007).

\section {Preliminaries on super algebraic geometry} 
This section is devoted to  recall some basic notions of super algebraic geometry.
\subsection{Superschemes as ringed spaces and scheme functor} 
Let $A=A_{\bar{0}}+A_{\bar{1}}$ be a super commutative algebra and $I_{A_{\bar{1}}}$ be the  ideal generated by $A_{\bar{1}}$. Denote by $A_{\bbz}$  the quotient $A/I_{A_{\bar{1}}}$ and by $\Spec A_{\bbz}$  the ordinary prime spectrum of $A_{\bbz}$. The affine superscheme  $\Spec{A}$ associated to $A$
is defined as the ringed space $(\Spec{A_{\bbz}}, \mathcal{O}_{\Spec{A}})$, where the structure sheaf $\mathcal{O}_{\Spec{A}}$ is given by $\Gamma(D(f),\mathcal{O}_{\Spec{A}})=A_f$ for any non-zero divisor $f \in A_{\bbz}$. By a \textit{superscheme}, we refer to locally ringed space  $X=(X_{\bbz},\mathcal{O}_{X})$, which is locally an affine superscheme. 
We say that $X_{\bbz}$ is the \textit{even reduced} scheme of $X$. Denote by $\iota: X_{\bbz} \hookrightarrow X$  the canonical embedding given by the quotient $ \iota^{\sharp}: \mathcal{O}_X \twoheadrightarrow \mathcal{O}_{X_{\bbz}}$ of sheaves over  $X_{\bbz}$.

Let $\mathrm{Al}_{\mathbf{k}}$ be the category of  super commutative algebras over $\mathbf{k}$.
By \textit{$\mathbf{k}$-functor} we mean a functor from $\mathrm{Al}_{\mathbf{k}}$ to the category $\mathrm{Set}$ of sets. 
For  $ A \in \mathrm{Al}_{\mathbf{k}}$, the \textit{affine superscheme functor} associated to $A$ refers to  the $\mathbf{k}$-functor 
$$\Spec A: \mathrm{Al}_{\mathbf{k}} \rightarrow \mathrm{Set};T \mapsto  \Hom_{\mathrm{Al}_{\mathbf{k}}}(A,T).$$
For an ideal $I \subset A$,  the closed subscheme $V(I)$ is defined as the functor $\Spec(A/I)$. The open subscheme $D(I)$, by definition, is  the functor 
$$ D(I): \mathrm{Al}_{\mathbf{k}} \rightarrow \mathrm{Set}, T \mapsto \{ \phi \in \Spec A(T) \mid \phi(I)T=T\}.$$

We say a $\mathbf{k}$-functor is covered by a collection of open subfunctors $\{X_i\}_{i \in I}$ if $X(\mathbf{f})=\cup_{i \in I} X_i(\mathbf{f})$ for any field extension $\mathbf{f} \supset \mathbf{k}$. A $\mathbf{k}$-functor $X$ is said to be a \textit{superscheme functor} if it is covered by a collection of open affine sub-superschemes.  

A functorial superscheme  determines a geometric superscheme in the apparent way suggested by the notation. 
Conversely,  a geometric superscheme $X$ defines a functorial superscheme 
by $X(T)=\Mor(\Spec T,X)$ for all $ T \in \mathrm{Al}_{\mathbf{k}}$. This is an equivalence 
between the categories of geometric and functorial superschemes. 

By a \textit{supervariety} $X$,  we mean a superscheme $X=(X_{\bbz},\mathcal{O}_X)$ whose even reduced scheme $(X_{\bbz},\mathcal{O}_{X_{\bbz}})$ is  reduced  and of finite type over $\mathbf{k}$. A \textit{group superscheme} is a superscheme $G$ such that $G(T)$ is a group for any $T \in \mathrm{Al}_{\mathbf{k}}$.
\begin{example}
	Let $\mathcal{A}$ be a finite-dimensional associative superalgebra over $\mathbf{k}$. We have the group superscheme $\Aut(\mathcal{A})$ of automorphisms of $\mathcal{A}$, which is, by definition, $\Aut(\mathcal{A})(T)=\Aut_{T}(\mathcal{A}\otimes_\mathbf{k} T)$ for $T \in\mathrm{Al}_{\mathbf{k}}$. 
\end{example}

\subsection{Smoothness and splitness of supervarieties}
Let $X=(X_{\bbz},\mathcal{O}_X)$ be a supervariety. The tangent sheaf $\mathrm{T}_X$ of $X$ is defined as the sheaf of derivations of $\mathcal{O}_X$. It is the sheaf associated to the presheaf given by $\Gamma(U,\mathrm{T}_X)=\mathrm{Der}_{\mathbf{k}}(\mathcal{O}_X(U))$
for any open subset $U$ of $X$. Let $\mathrm{T}_{X,x}$ denote the stalk of $\mathrm{T}_X$ at $x \in X_{\bbz}$ and   $\mathrm{Der}_x$ stand for the space
$$ \{ v: \mathcal{O}_{X,x} \rightarrow \mathbf{k} \mid v(fg)=v(f)g(x)+(-1)^{|v||f|}f(x)v(g) \text{ for all $f,g \in \mathcal{O}_{X,x}$}\}$$
of point derivations at $x$.  We say that $X$ is \textit{smooth} if the natural map $\mathrm{T}_{X,x} \rightarrow \mathrm{Der}_x$ is surjective.
Let $\hat{\mathcal{O}}_{X,x}$ be the completion  of $ \mathcal{O}_X$ at $x \in X_{\bbz}$ and $\mathbf{S}[[\mathrm{Der}_x^{*}]]$ be the algebra of formal power series of $\mathrm{Der}_x^{*}$. A superscheme is  smooth if and only if $\hat{\mathcal{O}}_{X,x}=\mathbf{S}[[\mathrm{Der}_x^{*}]]$. 
For the proof and other equivalent descriptions of smoothness, see \cite{Sh}. 

Given a super commutative algebra $A=A_{\bar{0}}+A_{\bar{1}}$, let $$\gr A=\bigoplus_{n\geq 0}I_{A_{\bar{1}}}^{n}/I_{A_{\bar{1}}}^{n+1}.$$ 
For a superscheme $X=(X_{\bbz},\mathcal{O}_X)$, $\gr X$ stands for the  superscheme obtained by gluing $\Spec(\gr \mathcal{O}_{X}(U))$, where $U$ runs over an affine open covering of $X_{\bbz}$.  A superscheme $X$ is said to be \textit{split} if $X$ is isomorphic to $\gr X$. In the case when $X$ is smooth, there exists a vector bundle $E$  over  $X_{\bbz}$  such that $\gr X$ is isomorphic to 
$(X_{\bbz}, \bigwedge^{\bullet} E)$. 
 
\section{Super $\mathcal{D}_X$-modules and de Rham cohomology}

 In this section, we investigate $\mathcal{D}_X$-mdoules, de Rham cohomology and jet bundles within the framework of super algebraic geometry and investigate some of their fundamental properties.

\subsection{Super $\mathcal{D}_X$-modules}
Let $X$ be a smooth superscheme of finite type over $\mathbf{k}$.
The ring $\mathcal{D}_{X}$ of differential operators  over $X$ is defined as the subalgebra of  $\mathcal{E}nd_{\mathbf{k}}(\mathcal{O}_X)$ generated by  $\mathrm{T}_X$ and $\mathcal{O}_X$. For any $p \in X_{\bbz}$, 
there exists an affine neighborhood $U \subset X_{\bbz}$ of $p$ and a set of local coordinates $ \{ x_1,\ldots, x_{n+m} \in \mathcal{O}_{X}(U) \}$ such that  $\{ \partial_{x_i} | 1 \leq i \leq n+m \}$ is a basis of $\mathrm{T}_X(U)$ as an $\mathcal{O}_X(U)$-module. Here $\partial_{x_i}$ is given by $\partial_{x_i}( x_j)=\delta_{i,j}$. 
Assume  that $x_1,\dots, x_n$ is even and $x_{n+1},\dots, x_{n+m}$ is odd. The odd coordinates 
$x_{i+n}$ are sometimes denoted by $s_{i}$  for convenience.  
In the above setting, we have 
$$\mathcal{D}_X(U)=\bigoplus_{\alpha \in \mathbb{N}^n; \beta \in \{0,1\}^{m}}\mathcal{O}_X(U)\partial_x^{\alpha}\partial_s^{\beta},$$ 
where  $\mathbb{N}^n$ denotes the set of  $n$-tuples of non-negative integers;
$\{0,1\}^{m}$ denotes the set of  $m$-tuples of $\{0,1 \}$; for $\alpha=(\alpha_1,\ldots \alpha_n) \in \mathbb{N}^n$, $\partial_x^{\alpha}=\partial_{x_1}^{\alpha_1} \cdots \partial_{x_n}^{\alpha_n}$ and   $\partial_s^{\beta}$ is defined similarly. Let $F^{\bullet}(\mathcal{D}_{X})$ be the order filtration of $\mathcal{D}_{X}$ given by
\begin{equation}
F^{p}(\mathcal{D}_{X})(U)=\{ f\partial_x^{\alpha}\partial_s^{\beta}| f \in \mathcal{O}_X(U), \sum_{i=1}^n\alpha_{i}+ \sum_{j=1}^m\beta_j \leq p\}.
\end{equation} 
The sheaf  $\gr\mathcal{D}_{X}$ of the associated graded algebra is isomorphic to the structure sheaf $\mathcal{O}_{\mathrm{T}^{*}_X}$ of the cotangent bundle $\mathrm{T}^{*}_X$.

A sheaf $\mathcal{E}$ of $\mathcal{O}_X$-module is said to be a left $\mathcal{D}_{X}$-module if it admits a left $\mathcal{D}_{X}$-action compatible with the $\mathcal{O}_X$-action. Namely, $\mathcal{O}_X$-module  $\mathcal{E}$ becomes a left $\mathcal{D}_X$-module if and only if
there is an action
$$\nabla:  T_X \otimes \mathcal{E} \rightarrow \mathcal{E}, (\theta,x) \mapsto \nabla_{\theta}\cdot x$$ 
such that 
\begin{itemize}
\item[(1)] $f \nabla_{\theta}\cdot x=f(\nabla_{\theta}\cdot x)$,
\item[(2)] $\nabla_{\theta}\cdot(fx)=\theta(f)x+(-1)^{|f||\theta|}f\nabla_{\theta}\cdot x$,
\item[(3)] $\nabla_{[\theta_1,\theta_2]}\cdot x=\nabla_{\theta_1} \cdot (\nabla_{\theta_2}\cdot x)-(-1)^{|\theta_1||\theta_2|}\nabla_{\theta_2} \cdot (\nabla_{\theta_1}\cdot x)$
\end{itemize}
hold for all local sections $f \in \Gamma(U,\mathcal{O}_X)$, $x \in \Gamma(U,\mathcal{E})$ and $\theta,\theta_1,\theta_2 \in \Gamma(U,\mathrm{T}_X)$.

Let $\Omega_{X}^{1}:=\mathcal{H}om_{\mathcal{O}_X}(\mathrm{T}_X, \mathcal{O}_X)$,  $\Pi$ be the parity reversing functor, and  $\Omega_{X}^{\bullet}$ be the sheaf 
of commutative graded algebras $\mathbf{S}(\Pi \Omega_{X}^{1})$. There is a canonical differential operator $\mathrm{d}$ on $\Omega_{X}^{\bullet}$  
given by 
$$ \mathrm{d}(f)=\sum_{i=1}^{n} \frac{\partial f}{\partial x_i} \mathrm{d}x_i, \mathrm{d}^{2}(f)=0,   \mathrm{d}(\alpha\beta)=\mathrm{d}(\alpha)\beta+(-1)^{|\alpha|}\alpha\mathrm{d}\beta$$ 
on the local sections.

We have $\mathrm{d}^{2}=0$ and call $(\Omega_{X}^{\bullet}, \mathrm{d})$ the \textit{de Rham complex} of $X$. For an $A \in \mathrm{Al}_{\mathbf{k}}$,  $\Omega_{A}^{\bullet}$ denotes  $\Gamma(\mathrm{Spec} A_{\bbz},\Omega_{\mathrm{Spec} A}^{\bullet})$.  Unlike that of schemes of finite type over $\mathbf{k}$, the de Rham complex of a superscheme of finite type  may have infinitely many non-zero terms. 

For a coherent $\mathcal{O}_X$-module $\mathcal{E}$, a \textit{connection} on $\mathcal{E}$ means a $\mathbf{k}$-linear map $\mathrm{d}_{\mathcal{E}}: \mathcal{E} \rightarrow  \Omega_{X}^{1}\otimes_{\mathcal{O}_X} \mathcal{E}$  such that  
$$ \mathrm{d}_{\mathcal{E}}(fx)=\mathrm{d}(f) \otimes x + (-1)^{|f|} f \otimes\mathrm{d}_{\mathcal{E}}x   $$ 
for all $f \in \Gamma(U,\mathcal{O}_X)$, $x \in \Gamma(U,\mathcal{E})$. There is a natural extension $\mathrm{d}_{\mathcal{E}}:  \Omega^{k}_{X}(\mathcal{E}) \rightarrow \Omega^{k+1}_{X}(\mathcal{E})$ given by 
$$ \mathrm{d}_{\mathcal{E}}(\alpha \otimes x)=\mathrm{d}\alpha\otimes x +(-1)^{|\alpha|}\alpha \otimes \mathrm{d}_{\mathcal{E}}x$$ 
for all $\alpha \in \Gamma(U,\Omega^{k}_{X})$ and $x \in \Gamma(U,\mathcal{E})$.
We say that $\mathrm{d}_{\mathcal{E}}$ is a \textit{flat connection} if $\mathrm{d}_{\mathcal{E}}^{2}=0$ and   $(\Omega^{\bullet}_X(\mathcal{E}), \mathrm{d}_{\mathcal{E}})$ is the \textit{de Rham complex} of $\mathcal{E}$.

For a coherent $\mathcal{O}_X$-module $\mathcal{E}$,  a left $\mathcal{D}_{X}$-modules structure   on $\mathcal{E}$ is equivalent to a left flat connection on $\mathcal{E}$ by the canonical isomorphism  
$$\Hom_{\mathcal{O}_X}(\mathcal{E}, \Omega^{1}_X \otimes_{\mathcal{O}_X} \mathcal{E})= \Hom_{\mathcal{O}_X}( \mathrm{T}_X \otimes_{\mathcal{O}_X} \mathcal{E}, \mathcal{E}).$$

For a Lie superalgebra $\mathfrak{g}$, let $\mathfrak{g}_{X}$ be the trivial vector bundle with fiber $\mathfrak{g}$. We regard $\mathfrak{g}_{X}$ as a left $\mathcal{D}_X$-module in the obvious  way. The derivation   
$\mathrm{d}_{\mathfrak{g}_X}$ of the complex $\Omega_{X}^{\bullet}( \mathfrak{g}_X)$ is also denoted by $\mathrm{d}$ for simplicity. By a $\mathfrak{g}$-valued $k$ form, we refer to a section of $\Omega_{X}^{k}(\mathfrak{g}_X) $.  For $\alpha u \in \Omega_{X}^{k}(\mathfrak{g}_X)$ and  $\beta v \in \Omega_{X}^{l}(\mathfrak{g}_X)$, let $\alpha u \wedge \beta v:= \alpha \wedge \beta[u,v] \in  \Omega_{X}^{k+l}(\mathfrak{g}_X)$.

\subsection{Hypercohomology and de Rham cohomology}
Let $\mathscr{A},\mathscr{B}$ be a pair of abelian categories and  $F: \mathscr{A} \rightarrow \mathscr{B} $ be a right exact covariant functor.
 Assume that $\mathscr{A}$ has 
enough injective objects.  Denote by $RF$ the right derived functor of $F$ and by $ \mathbf{R}F: D^{+}(\mathscr{A}) \rightarrow D^{+}(\mathscr{B})$ the corresponding functor of derived categories. For a bounded below complex $C^{\bullet}$  of $\mathscr{A}$, the hypercohomology $\mathbb{H}^{i}_F(C^{\bullet})$ is defined as the $i$-th cohomology of the complex $\mathbf{R}F(C^{\bullet})$. Hypercohomology can be computed by the spectral sequence
\begin{equation}\label{spec_sequ}
E_2^{pq}= RF^{p}(H^q(C^{\bullet}))\Rightarrow \mathbb{H}^{p+q}_F(C^{\bullet}).
\end{equation} 

We list some elementary properties of hypercohomology as follows.
For a single object $ C \in \mathscr{A}$ which is  seen  as a complex concentrated at the $0$ degree,   $\mathbb{H}^{i}_F(C)$ coincides with  derived functor $RF^{i}(C)$.
Quasi-isomorphisms of complexes preserve the hypercohomology. A short exact sequence of complexes gives a long exact sequence of  hypercohomology.

Now let $\mathscr{A}$  and $\mathscr{B}$ be the category of sheaves over $X$  and $F=\Gamma(X,\text{-})$
be the functor that takes global section. In this case, we simplify $\mathbb{H}^{i}_F(\text{-})$ by $\mathbb{H}^{i}(\text{-})$.
The \textit{de Rham cohomology} $H_{\mathrm{dR}}^{i}(X)$ of $X$, by definition,  is the hypercohomology $\mathbb{H}^{i}(\Omega_{X}^{\bullet})$. The de Rham cohomology $H_{\mathrm{dR}}^{i}(X,\mathcal{E})$ of a $\mathcal{D}_X$-module $\mathcal{E}$ is defined to be $\mathbb{H}^{i}(\Omega_{X}^{\bullet}(\mathcal{E}))$.
We have that $H_{\mathrm{dR}}^{i}(X)=0$ for all $ i > 2\dim(X_{\bbz}) $. However,
  unlike  the sheaf cohomology, $H_{\mathrm{dR}}^{i}(X)$ may not vanish  for $i$ with $\dim(X_{\bbz})< i \leq 2\dim(X_{\bbz})$.  Another significant difference is that the  de Rham cohomology of an affine (super)scheme  may not vanish.

By the same argument as in the even case \cite{St}, we have 
\begin{prop}[K$\ddot{\text{u}}$nneth formula]\label{KUN-FOR} 
Let $X,Y$ be a pair of  smooth superschemes over $\mathbf{k}$. Then  
$$H^{k}_{\mathrm{dR}}(X \times Y)= \bigoplus_{i+j=k} H^{i}_{\mathrm{dR}}(X) \otimes_\mathbf{k} H^{j}_{\mathrm{dR}}(Y).$$	
\end{prop}

\subsection{Extensions of  $\mathcal{D}_X$-modules}
For a pair of $\mathcal{D}_X$-modules $\mathcal{E}_1$ and $\mathcal{E}_2$,  denote by  
$\mathcal{H}\mathrm{om}_{\mathcal{D}_X}(\mathcal{E}_1,\mathcal{E}_2)$ the sheaf of $\mathcal{D}_X$-module homomorphisms 
from $\mathcal{E}_1$ to $\mathcal{E}_2$. 
The space $\mathrm{Hom}_{\mathcal{D}_X}(
\mathcal{E}_1,\mathcal{E}_2)$ of homomorphisms from $\mathcal{E}_1$ to $\mathcal{E}_2$ coincides with  the global section of $\mathcal{H}\mathrm{om}_{\mathcal{D}_X}(\mathcal{E}_1,\mathcal{E}_2)$. Denote by $\mathcal{E}\mathrm{xt}^{\bullet}_{\mathcal{D}_X}(\mathcal{E}_1, \text{-})$ and  $\Ext^{\bullet}_{\mathcal{D}_X}(\mathcal{E}_1, \text{-})$  the corresponding right derived functors of $\mathcal{H}\mathrm{om}_{\mathcal{D}_X}(\mathcal{E}_1, \text{-})$ and $\mathrm{Hom}_{\mathcal{D}_X}(\mathcal{E}_1, \text{-})$, respectively. The following  theorem relates $\mathcal{E}\mathrm{xt}^{\bullet}_{\mathcal{D}_X}(\mathcal{E}_1, \text{-})$ and  $\Ext^{\bullet}_{\mathcal{D}_X}(\mathcal{E}_1, \text{-})$.
\begin{theorem}\label{spec_dmod}
For any quasi-coherent left $\mathcal{D}_{X}$-modules $\mathcal{E}_1$ and $\mathcal{E}_2$, 	
we have  the spectral sequences 
$$E_{2}^{pq}=H^{p}(X,\mathcal{E}xt^{q}_{\mathcal{D}_X}(\mathcal{E}_1,\mathcal{E}_2)) \Rightarrow \Ext_{\mathcal{D}_X}^{p+q}(\mathcal{E}_1,\mathcal{E}_2).$$
\end{theorem}
\begin{proof}
Let $F: \mathscr{A} \rightarrow \mathscr{B} $ and $G: \mathscr{B} \rightarrow \mathscr{C}$ 
be a pair of  morphisms of abelian categories. Assume that any injective object $J$ of $\mathscr{A}$ is $G$-acyclic,
namely, $R^{i}G(F(J))=0$ for all $i>0$.
In this case, we have the following  spectral sequence due to Grothendick    
$$E_2^{pq}=RG^{p}(RF^{q}(\mathcal{E})) \Rightarrow R^{p+q}(G\circ F)(\mathcal{E}).$$
Now let $\mathscr{A}=\mathscr{B}$ be the category of $\mathcal{D}_X$-modules, $\mathscr{C}$ be category of abelian groups, $F=\mathcal{H}om_{\mathcal{D}_X}(\mathcal{E}_1,\text{-})$ and $G=R\Gamma(X,\text{-})$.	
Note that $\Gamma(X,\mathcal{H}om_{\mathcal{D}_X}(\mathcal{E}_1,\mathcal{E}))=\Hom_{\mathcal{D}_X}  (\mathcal{E}_1,\mathcal{E})$.  By  the same argument as in the case of $\mathcal{O}_X$-modules, we can show that   $\mathcal{H}om_{\mathcal{D}_X}(\mathcal{E}_1,\text{-})$
sends quasi-coherent injective $\mathcal{D}_X$-modules to flasque sheaves, which are $G$-acyclic.  The theorem follows from the Grothendick's 
 spectral sequence.
\end{proof}

Let $\mathrm{Sp}^{\bullet}_{X}$ be the Spencer complex, which is defined 
as follows;  $\mathrm{Sp}^{i}_{X}:=\mathcal{D}_X \otimes_{\mathcal{O}_X} \mathbf{S}^{i} (\Pi(\mathrm{T}_X))$, the differential $\delta: \mathrm{Sp}^{i}_{X}  \rightarrow  \mathrm{Sp}^{i-1}_{X}$ at the $i$ degree is  given by
\begin{align*}
	\delta(P \otimes \theta_1 \theta_2 \cdots  \theta_n )=&\sum_{k} (-1)^{n_{\theta_i}}P\theta_i \otimes   \theta_1 \cdots  \hat{\theta}_k \cdots  \theta_n +  \\
	&\sum_{k<l} (-1)^{n_{\theta_k}+n_{\theta_l}+|\theta_k||\theta_l|} P \otimes [\theta_k, \theta_l] \theta_1 \cdots  \hat{\theta}_k \cdots \hat{\theta}_l  \cdots \theta_n.
\end{align*}
Here, the integer $n_{\theta_i}<n$ is defined  as follows. If $|\theta_k|=0$, $n_{\theta_i}=0$; if 
$|\theta_k|=1$, then $n_{\theta_k}=\sum_{i=1}^{k-1}|\theta_k||\theta_i|$. Let $\epsilon: \mathrm{Sp}^{0}_{X}=\mathcal{D}_X \rightarrow \mathcal{O}_X $ be the $\mathcal{D}_X$-module homomorphism given by $\epsilon(P):=P\cdot1$ for each local section $P$ of $\mathcal{D}_X$.
 
\begin{prop}[Spencer resolution]
The  complex 
$\begin{tikzcd}[cramped, sep=small]
	\mathrm{Sp}^{\bullet}_{X}  \arrow[r,  "\epsilon" ] &  \mathcal{O}_X \arrow[r]  & 0
\end{tikzcd}$ 
is a locally free resolution  of the  $\mathcal{D}_X$-module $\mathcal{O}_X$. 
\end{prop}
  \begin{proof}
The  degree filtration of $\mathrm{Sp}^{\bullet}_{X}$  is exactly the super Kozsul complex, which is exact by \cite{No}.
\end{proof}

\begin{prop} \label{extension of D-mods}
For any left coherent $\mathcal{D}_{X}$-module $\mathcal{E}$, we have 
$$H^{i}_\mathrm{{dR}}(X, \mathcal{E})=\mathrm{Ext}^{i}_{\mathcal{D}_{X}}(\mathcal{O}_X,\mathcal{E})$$
for $i\geq 0$.
\end{prop}
\begin{proof}
Since $\mathrm{Sp}^{\bullet}_X$ is a locally free resolution of $\mathcal{O}_X$ as $\mathcal{D}_{X}$-module, we have 
$$\mathcal{E}xt^{i}_{\mathcal{D}_{X}}(\mathcal{O}_X,\mathcal{E})= \mathcal{H}^{i}(\mathcal{H}om_{\mathcal{D}_{X}}(\mathrm{Sp}^{\bullet}_X,\mathcal{E})).$$ Note that $\mathcal{H}om_{\mathcal{D}_{X}}(\mathrm{Sp}^{\bullet}_X,\mathcal{E})=\Omega_{X}^{\bullet}(\mathcal{E})$.	 
The proposition follows from Theorem \ref{spec_dmod} and the spectral sequence \eqref{spec_sequ}.
\end{proof}

The following theorem  is due to Polishchuk. 
\begin{theorem}[\cite{Po}, Theorem 1.1]\label{equi_dR_HOM} 
Let $X=(X_{\bbz},\mathcal{O}_X)$ be a smooth superscheme.	
There is an isomorphism 
$H^{\bullet}_{\mathrm{dR}}(X) \rightarrow H^{\bullet}_{\mathrm{dR}}(X_{\bbz})$.
\end{theorem}
We give an different proof here since the construction of isomorphism  will be needed  later.
\begin{proof}
By the spectral sequence \eqref{spec_sequ}, we only need to prove 
  $$\mathcal{H}^{i}(\Omega_{X}^{\bullet})=\mathcal{H}^{i}(\Omega_{X_{\bbz}}^{\bullet}).$$  
This argument is local, so we may assume that $X=X_{\bbz}\times\mathbf{k}^{(0|m)}$. 
The theorem follows from the K\"uneth formula (Proposition \ref{KUN-FOR})  and the fact that $H_{\mathrm{dR}}^0(\mathbf{k}^{(0|m)})=\mathbf{k}$, and $H_{\mathrm{dR}}^i(\mathbf{k}^{(0|m)})=0$ for all $i>0$.	 
\end{proof}

This theorem has a categorical explanation as follows.
The inclusion $\iota: X_{\bbz} \hookrightarrow X $ induces a pull-back 
$$ \iota^{*}: \mathcal{D}_{X}\text{-}\mathrm{mod} \rightarrow \mathcal{D}_{X_{\bbz}}\text{-}\mathrm{mod}$$ 
of $\mathcal{D}$-modules.  Penkov proved that   $\iota^{*}$ is an equivalence of categories in the holomorphic setting \cite{Pe2}. His argument is valid  here as well. Therefore,  Proposition \ref{extension of D-mods} implies  Theorem \ref{equi_dR_HOM}.

From now on, we will encounter  pro-superschemes  frequently. 
As in the even case\cite{BK}, we may regard them as usual superschemes.
\subsection{Jet bundles}   
Let $\Delta$ be the diagonal of $X^2:=X\times X$,  $\hat{\Delta}$ be the completion of  $X^{2}$ along  $\Delta$ and $\pi_1$, $\pi_2$ be the projection of 
$\hat{\Delta}$ onto the first and second factor. The jet bundle $J^{\infty}(\mathcal{E})$ of a vector bundle $\mathcal{E}$ over $X$ is defined as $\pi_{1 *}\pi_2^*(\mathcal{E})$.  Completing the relative de Rham differential  $\mathrm{d}_{X^2/X}: \mathcal{O}_{X^2} \rightarrow \Omega^1_{X^2/X}$ along $\Delta$, we have the Grothendick connection $\nabla: J^{\infty}(\mathcal{O}_X) \rightarrow \Omega^1_X \otimes_{\mathcal{O}_X} J^{\infty}(\mathcal{O}_X)$. We also have a flat connection on $J^{\infty}(\mathcal{E})=J^{\infty}(\mathcal{O}_X)\otimes_{\mathcal{O}_X} \mathcal{E} $ by taking the tensor product of  $\nabla$ and the trivial connection on $\mathcal{E}$.
Let $J^{\infty}(\Omega^{\bullet}_{X}(\mathcal{E}))$ stand for the associated de Rham complex. 
In the case of $X=\Spec A$ being an affine superscheme, the above notions can be described explicitly as follows
 \begin{align*}
J^{\infty}(\mathcal{O}_X)&= \widehat{A\otimes_{\mathbf{k}} A};	\\
\Omega_X^{1}\otimes_A J^{\infty}(\mathcal{O}_X)&=\Omega_X^{1}\otimes_A \widehat{A\otimes_{\mathbf{k}} A}=\widehat{\Omega_X^{1}\otimes_{\mathbf{k}} A}; \\
\nabla(a\otimes b)&=\mathrm{d}(a)\otimes b \text{ \  for all $a\in A$ and $b \in J^{\infty}(\mathcal{O}_X)$ or $J^{\infty}(\mathcal{E})$}.
\end{align*}	
Here, for an $A \otimes A$-module $M$, $\widehat{M}$ denotes the completion of $M$  with respect to the ideal $I$ corresponding to $\Delta$.  

We have the following generalization of [Theorem 4.4,\cite{Ye2}] to the super case.
\begin{prop}\label{de Rham to Cech}
	For a quasi-coherent sheaf $\mathcal{E}$ on $X$,  the complex $J^{\infty}(\Omega^{\bullet}_{X}(\mathcal{E}))$ is a resolution of $\mathcal{E}$ as $\mathcal{O}_X$-modules, i.e,  $J^{\infty}(\Omega^{\bullet}_{X}(\mathcal{E}))$ is exact.
\end{prop}
\begin{proof}
	We use the same argument as in even case  \cite{Ye2}. 
 Since the statement is local, we may assume that $X=\Spec A$ and $X$ admits an $\acute{\text{e}}$tale coordinate system 
	$(x_1,\ldots,x_n)$ with $\mathcal{O}_X=\mathcal{O}_{X_{\bbz}} \otimes_{\mathbf{k}} \bigwedge^{\bullet} (x_{n+1},\ldots,x_{n+m})$. Let $\mathbf{x}$ be the corresponding $\acute{\text{e}}$tale coordinate system $ \{ x_1,\ldots,x_n \mid x_{n+1},\ldots,x_{n+m} \}$. 
	
  Then we have a natural isomorphism 
 $A \otimes_{\mathbf{k}[\mathbf{x}]} \Omega^{\bullet}_{\mathbf{k}[\mathbf{x}]}=\Omega^{\bullet}_{A}$ 	 and hence 
 \begin{equation}\label{equ 3.3}
 \Omega^{\bullet}_{\mathbf{k}[\mathbf{x}]}\otimes_{\mathbf{k}[\mathbf{x}]} A\otimes_{\mathbf{k}}A=  \Omega^{\bullet}_{A}\otimes_{\mathbf{k}} A.
 \end{equation}
 
We have $\hat{\mathcal{O}}_{X,x}=\mathbf{k}[[\mathbf{x}]]$ and $ \widehat{A\otimes_{\mathbf{k}} A}=A[[\mathbf{y}]]$ as completed algebras. Here $\mathbf{y}$ stands for the coordinate system  given by  $y_i=x_i\otimes 1 - 1 \otimes x_i \in A\otimes_{\mathbf{k}} A $.
	 Completing the two sides of \eqref{equ 3.3} with respect to the $I$-adic topology, we get an isomorphism 
	$$\Omega^{\bullet}_{\mathbf{k}[\mathbf{x}]}  \otimes_{\mathbf{k}[\mathbf{x}]} A[[\mathbf{y}]]  \longrightarrow \widehat{\Omega_{A}^{\bullet} \otimes_{\mathbf{k}} A}.$$
The first complex is clearly  a resolution of $\mathcal{O}_{X}$.  
Finally, taking the tensor product of  $J^{\infty}(\Omega^{\bullet}_{X})$  and  $\mathcal{E}$, we complete the proof.
\end{proof}

Since hypercohomology is invariant under the quasi-isomorphism of complexes, Proposition \ref{de Rham to Cech} implies
\begin{corollary} \label{coro_3.7} 
	For each vector bundle $\mathcal{E}$ on $X$, we have 
	$$\mathrm{H}^{i}_{\mathrm{dR}}(X, J^{\infty}(\mathcal{E}))=H^{i}(X,\mathcal{E})$$ 
	for all $i \geq 0$.
\end{corollary}

\subsection{Hodge filtration}

Let $\sigma_{ \geq n}\Omega^{\bullet}_{X}$ be the subcomplex
of $\Omega^{\bullet}_X$ defined by  $(\sigma_{ \geq n}\Omega^{\bullet}_{X})^i=(\Omega^{\bullet}_{X})^{i}$ if $i \geq n$ and $(\sigma_{\geq n}\Omega^{\bullet}_{X})^i=0$ otherwise.
The natural short exact sequence
$$0 \rightarrow \sigma_{\geq 1}\Omega^{\bullet}_{X} \rightarrow \Omega^{\bullet}_{X} \rightarrow  \mathcal{O}_X \rightarrow 0  $$  induces a long exact sequence
 \begin{equation} \label{les3.4}
 \cdots \rightarrow  H_{F}^{i}(X) \rightarrow H_{\mathrm{dR}}^{i}(X) \rightarrow H^{i} (X,\mathcal{O}_X) \rightarrow \cdots
\end{equation}
of hypercohomology groups.
Here $H_{F}^{i}(X)$ denotes the hypercohomology $\mathbb{H}^{i}(\sigma_{\geq 1}\Omega^{\bullet}_{X})$.
We have a similar long exact sequence for the even part $X_{\bbz}$ and the following commutative diagram
\begin{equation}\label{equ-surjec}
	\begin{CD}
\xymatrix{ H_{F}^{i}(X) \ar[r] \ar[d]  &
	 H_{\mathrm{dR}}^{i}(X)  \ar[r] \ar[d] & H^{i} (X,\mathcal{O}_{X}) \ar[d]\\
	 H_{F}^{i}(X_{\bbz}) \ar[r] & H_{\mathrm{dR}}^{i}(X_{\bbz}) \ar[r] & H^{i} (X_{\bbz},\mathcal{O}_{X_{\bbz}})}
	 \end{CD}
 \end{equation}
for all $i \geq 0$. Here the vertical arrows are induced by the pull-back   
$\iota^{*}: \Omega^{\bullet}_{X} \rightarrow \Omega^{\bullet}_{X_{\bbz}}$  of $\iota: X_{\bbz} \rightarrow X$. We say  that the supervariety $X$ is \textit{admissible} if the map 
$H_{\mathrm{dR}}^{i}(X) \rightarrow H^{i}(X,\mathcal{O}_X)$ is surjective for $i=1, 2$. Smooth affine supervarieties are admissible.

\begin{prop}\label{prop_admi} We consider the diagram \eqref{equ-surjec}.
\begin{itemize}	
\item[(\rmnum{1})] If the third vertical map is bijective (resp. injective) for each $i \geq 0$, then the first vertical map is bijective (resp. injective). 
\item[(\rmnum{2})] Assume that the  supervariety $X$ and its reduced part $X_{\bbz}$ are admissible. Then the kernels of horizontal arrows in the second column are $H_{F}^{i}(X)$ and $H_{F}^{i}(X_{\bbz})$ for $i=2,3$, respectively.  The first vertical map is injective for $i=2$, and the third one is surjective for $i=1,2$.
\item[(\rmnum{3})]  Assume that the supervariety $X$ is  split and admissible. Then its even reduced variety $X_{\bbz}$ is admissible.
\end{itemize}	 
\end{prop}

\begin{proof}
 (\rmnum{1}) The statement follows from the five-lemma (resp. four-lemma). (\rmnum{2}) The first  statement  is obvious. The second one follows from the first one, Theorem \ref{equi_dR_HOM} and commutativity of \eqref{equ-surjec}. (\rmnum{3}) Since X is split, the natural exact sequence $ \mathrm{Ker}(\iota^{\#}) \hookrightarrow \mathcal{O}_X \twoheadrightarrow \mathcal{O}_{\bbz}$ splits. Thus the third vertical map is surjective and (\rmnum{3}) follows from commutativity of \eqref{equ-surjec}.
\end{proof}

\section{Harish-Chandra torsors  and their extensions} 
\subsection{Harish-Chandra torsors}
Let $G$ be an algebraic supergroup and  $\mathcal{M}$ be a superscheme over $X$ with a morphism $\rho: \mathcal{M} \rightarrow X$.   
We say that $\mathcal{M}$  is a $G$-torsor if there is  a free action of the supergroup $G$ on $\mathcal{M}$  and  the action map 
$$G(T) \times \mathcal{M}(T) \rightarrow \mathcal{M}(T) \times_{X(T)}\mathcal{M}(T): (g,x) \mapsto (x,g\cdot x), \text{ $T \in \mathrm{Al}_{\mathbf{k}}$} $$ 
induces an isomorphism of superschemes over $X$.  A $G$-torsor is also called a principal $G$-bundle. In this paper,   we assume that all torsors are locally trivial in the Zariski topology.

Let $V$ be a $G$-module and $\mathcal{M}$ be a $G$-torsor  over $X$. The 
\textit{localization} (or \textit{associated bundle}) $\mathcal{V}=\mathrm{Loc}(\mathcal{M},V)$ of $V$ is a locally free sheaf over $X$ defined as follows. For an open sub-superscheme  $U \subset X$, the section $\Gamma(U,\mathcal{V})$  takes the $\mathcal{O}_U$-module of $G$-equivariant morphisms from $\rho^{-1}(U)$ to $V$. 
If there exists  a local trivialization $ \rho^{-1}(U)=U \times G \rightarrow U$, then  $\Gamma(U,\mathcal{V})=\mathbf{k}[U]\otimes_{\mathbf{k}} V$.  Given a supergroup homomorphism $f: G \rightarrow G_1$, we have a $G$-action on $G_1$ by 
$$  G(T)  \times G_1(T)  \rightarrow  G_1(T): (g, x) \mapsto  f(g)x \quad $$
for all $T \in \mathrm{Al}_{\mathbf{k}}, g \in G(T)$ and $ x \in G_1(T)$.  Note that the diagonal 
$G$-action on $\mathcal{M} \times G_1$ is free. It is easy to see that the GIT quotient $f_{*}(\mathcal{M}):=\mathcal{M} \times G_1 \sslash G$ is a $G_1$-torsor over $X$.  For a $G_1$-module $V$, we have  $\mathrm{Loc}(f_{*}(\mathcal{M}),V)=\mathrm{Loc}(\mathcal{M},f^{*}V)$.

Let $ \mathfrak{h}$ be a Lie superalgebra  that  contains $\Lie(G)$.  We say that
$\langle G, \mathfrak{h}  \rangle$ is a \textit{Harish-Chandra pair} if  $G$ acts on  $\mathfrak{h}$ in such a way that its differential action coincides with the adjoint action of $\mathrm{Lie}(G)$ on $\mathfrak{h}$. Denote by 
$$0 \rightarrow \xymatrix{ \mathfrak{g}_{\mathcal{M}} \ar[r]^{\iota_{\mathcal{M}}}& \mathcal{E}_{\mathcal{M}}} \rightarrow \mathrm{T}_X \rightarrow 0  $$ 
 the Atiyah extension on $X$. Here $\mathfrak{g}_{\mathcal{M}}$ is the associated bundle of the $G$-module $\mathfrak{g}=\mathrm{Lie}(G)$. A bundle map 
$\theta_{\mathcal{M}}: \mathcal{E}_{\mathcal{M}} \rightarrow \mathfrak{h}_{\mathcal{M}}$
is called $G$-equivariant connection if $\theta_{\mathcal{M}}\circ\iota_{\mathcal{M}}$ coincides with the embedding $\mathfrak{g}_{\mathcal{M}} \hookrightarrow \mathfrak{h}_{\mathcal{M}}$. The  connection $\theta_{\mathcal{M}}$ is called  \textit{flat} if the $\mathfrak{h}$-valued 1-form $ \alpha=\rho^{*}(\theta_{\mathcal{M}})$ on $\mathcal{M}$ satisfies $2\mathrm{d}\alpha + \alpha\wedge\alpha=0$. By a \textit{Harish-Chandra  $\langle G, \mathfrak{h} \rangle$-torsor} $\mathcal{M}$, we refer to a $G$-torsor $\mathcal{M}$  equipped with a flat connection $\theta_{\mathcal{M}}$. A Harish-Chandra  $\langle G, \mathfrak{h} \rangle$-torsor $\mathcal{M}$ is called \textit{transitive}, if $\theta_{\mathcal{M}}$ is a bundle isomorphism. 
  
Let  $V$ be an $\mathfrak{h}$-module with a $G$-action. We say that $V$ is a \textit{Harish-Chandra} $\langle G, \mathfrak{h}  \rangle$-\textit{module} if the differential  of the $G$-action coincides with the restriction of $V$ on  $\mathfrak{g} \subset \mathfrak{h}$. For simplicity,
by a $\langle G, \mathfrak{h}  \rangle$-module (resp. torsor), we always mean a Harish-chandra $\langle G, \mathfrak{h}  \rangle$-module (resp. torsor).

Let $(\mathcal{M},\theta_{\mathcal{M}})$ be a  $\langle G, \mathfrak{h} \rangle$-torsor and $V$ be $\langle G, \mathfrak{h}  \rangle$-module. For a section $s=fv \in \Gamma(U,\mathcal{V})$ 
with $ f \in \mathcal{O}_{\mathcal{M}}(\rho^{-1}(U))$ and $v \in V$, $\xi \in 
\Gamma(U,\mathcal{E}_{\mathcal{M}})$, set 
$$
\nabla^{\mathcal{M}}_{\xi}(fv)=\xi\cdot fv+(-1)^{|f||\xi|}f \theta_{\mathcal{M}}(\xi)v.
$$
It is clear that  $\nabla^{\mathcal{M}}_{\xi}s=0$ for all $s\in \Mor_{G}(\rho^{-1}(U), V)$ and $ \xi \in \Gamma(U, \iota_{\mathcal{M}}(\mathfrak{g}_{\mathcal{M}}))$. Hence $\nabla^{\mathcal{M}}$  induces  a flat connection on  $\mathcal{V}$, which is denoted by  $\nabla$.

Let $\mathrm{Loc}(\mathcal{M},\text{-})$ be the localization functor which sends each $\langle G, \mathfrak{h} \rangle$-module $V$ to the vector bundle $\mathcal{V}=\mathrm{Loc}(\mathcal{M},V)$.  By Proposition \ref{extension of D-mods} and the fact $\mathrm{Ext}_{\langle G, \mathfrak{h} \rangle}^{i}(\mathbf{k}, V)=H^{i}(\langle G, \mathfrak{h} \rangle,V)$,  we have a morphism 

$$\mathrm{Loc}(\mathcal{M},\text{-}):  H^{i}(\langle G, \mathfrak{h} \rangle,V) \longrightarrow H_{\mathrm{dR}}^{i}(X,\mathcal{V}).$$  
Let  
\begin{equation}\label{ses}
 0 \rightarrow U \rightarrow V \rightarrow W \rightarrow 0 
\end{equation} 
be  a short exact sequence of $\langle G, \mathfrak{h} \rangle$-modules.  Their localizations $\mathcal{U}$, $\mathcal{V}$ and $\mathcal{W}$ form a short exact sequence.  We have the associated  long exact sequence
\begin{equation}\label{les}
\xymatrix{ H_{\mathrm{dR}}^{i-1}(X,\mathcal{W}) \ar[r]^{d_{i}} & H_{\mathrm{dR}}^{i}(X,\mathcal{U}) \ar[r]^{a_{i}} & H_{\mathrm{dR}}^{i}(X, \mathcal{V}) \ar[r]^{b_{i}} &  H_{\mathrm{dR}}^{i}(X,\mathcal{W})}	
\end{equation}	 
of de Rham cohomology groups for all $i>0$.

\subsection{Recall: groupoid $\mathcal{S}pl(c)$} \label{sec4.2}
In this subsection, we recall the construction and some basic properties of $\mathcal{S}pl(c)$ 
given in [\S5, \cite{BK}].

Let $A,B$ be two objects in an  abelian category $\mathscr{C}$. Denote by $\mathcal{E}xt^{1}(B,A)$ the extension groupoid.  Fix $c \in \mathrm{Ext}^{2}(B,A)$  and its Yoneda presentation
\begin{equation}\label{equ_yoneda}
\begin{tikzcd}[cramped, sep=small] 
	0 \arrow[r] & A \arrow[r] & E_1 \arrow[r] & E_2 \arrow[r] & B \arrow[r] & 0.
\end{tikzcd}
\end{equation}
Let $\mathcal{S}pl(c)$ stand for the groupoid consist of the exact sequences in the form of 
$$\begin{tikzcd}[cramped, sep=small]
A  \arrow[r,  "f" ] & E \arrow[r,"g"]  & B
\end{tikzcd}$$ 
such that $ E_1 \cong \Ker g $, $E_2 \cong \mathrm{Coker} f $ and  the natural sequence
\begin{equation*} 	
\begin{tikzcd}[cramped, sep=small] 
	0 \arrow[r] & A \arrow[r] &  \Ker g \arrow[r] &  \mathrm{Coker} f \arrow[r] & B \arrow[r] & 0
\end{tikzcd}
\end{equation*}
is isomorphic to \ref{equ_yoneda}. The groupoid $\mathcal{S}pl(c)$ is nontrivial if and only if $ c \in \mathrm{Ext}^{2}(B,A)$ is trivial. There are a pair of sum and difference operations
$$ \tilde{+} :\mathcal{E}xt^{1}(B,A) \times \mathcal{S}pl(c) \rightarrow \mathcal{S}pl(c); \quad  \tilde{-}:  \mathcal{S}pl(c) \times \mathcal{S}pl(c) \rightarrow \mathcal{E}xt^{1}(B,A) $$
given by the Baer sum.  The sum $\tilde{+}$ defines a transitive $\mathcal{E}xt^{1}(B,A)$-action on 
$\mathcal{S}pl(c)$.

Let 
$$\begin{tikzcd}[cramped, sep=small] 
0 \arrow[r] & A_0 \arrow[r,"a"] & A  \arrow[r,"b"] & A_1  \arrow[r] & 0 
\end{tikzcd}$$
be a short exact sequence in 
$\mathscr{C}$. For a fixed $B \in \mathscr{C}$,  we have the exact sequence 

\begin{equation} \label{equ_4.4} 
\begin{tikzcd}[cramped, sep=small] 
	\Ext^{i-1}(B,A_1)  \arrow[r,"\tilde{d}_i"] &  \Ext^{i}(B,A_0) \arrow[r,"\tilde{a}_i"] &  \Ext^{i}(B,A) \arrow[r,"\tilde{b}_i"] &  \Ext^{i}(B,A_1)
\end{tikzcd}	
\end{equation}
for all $i \geq 1$.  Assume that $\tilde{b}_2(c)=0$. Each  $s \in \mathcal{S}pl(\tilde{b}_2(c))$ canonically defines a class $c_{0} \in \Ext^{2}(B,A_0)$ such that $\tilde{a}_{2}(c_0)=c$. For any $e \in \Ext^{1}(B,A_1)$,  the twist $ e\tilde{+}s \in \mathcal{S}pl(b_2(c))$, by construction, corresponds to $c_0+\tilde{d}_1(e)$.

\subsection{Extensions of Harish-Chandra torsors }
Let $V$ be a  $ \langle G, \mathfrak{h} \rangle$-module and 
\begin{equation}\label{extension}
\begin{tikzcd}[cramped, sep=small] 
1 \arrow[r]& V \arrow[r,"\sigma"]& \langle G_1,\mathfrak{h}_1\rangle \arrow[r,"\pi"]&  \langle G, \mathfrak{h} \rangle \rightarrow 1
\end{tikzcd}
\end{equation}
be an  extension of  $\langle G, \mathfrak{h} \rangle$ 
by $V$. For a transitive $\langle G, \mathfrak{h} \rangle$-torsor $\mathcal{M}$ over $X$, denote by  $H_{\mathcal{M}}^{1}(X,\langle G_1, \mathfrak{h}_1) \rangle$  the set of isomorphism classes 
of $\langle G_1, \mathfrak{h}_1 \rangle$-torsor $\mathcal{M}_1$  with $ \pi^{*}(\mathcal{M})=\mathcal{M}_1$. Note that the torsors in $H_{\mathcal{M}}^{1}(X,\langle G_1, \mathfrak{h}_1) \rangle$ are automatically transitive.  The goal of this  subsection is to describe $H_{\mathcal{M}}^{1}(X,\langle G_1, \mathfrak{h}_1) \rangle$.

 There  is a super version of Hochschild-Serre spectral sequences \cite{Mu} 
$$ E_{pq}^{2}=H^{p}(\langle G, \mathfrak{h} \rangle,H^{q}(V,V)) \Rightarrow  H^{p+q}(\langle G_1,\mathfrak{h}_1\rangle,V),$$ 
where the first $V$ in $H^{q}(V,V)$ denotes the Harish-Chandra pair $\langle V,V  \rangle$. 
The following is a piece of its lower degree
$$\begin{tikzcd}[cramped, sep=small]
H^{1}(V,V)^{\langle G, \mathfrak{h} \rangle} \arrow[r,"d"] & H^{2}(\langle G, \mathfrak{h} \rangle,V) \arrow[r,"\pi^*"] & H^{2}(\langle G_1,\mathfrak{h}_1\rangle,V). 
\end{tikzcd}$$
Let  $\tau_{V} \in H^{1}(V,V)$ be the  tautological class. Then 
$$c=d(\tau_V) \in H^2(\langle G, \mathfrak{h} \rangle,V)=\Ext_{\langle G, \mathfrak{h} \rangle}^2(\mathbf{k},V)$$ is the  class representing the extension
\eqref{extension}. The pull-back $\pi^*(c) \in \Ext_{\langle G_1, \mathfrak{h}_1 \rangle}^2(\mathbf{k},V)$ is trivial. The splittings of $\pi^*(c)$ for a fixed Yoneda representation form a  groupoid  $\mathcal{S}pl(\pi^*(c))$  over  $\Ext_{\langle G_1, \mathfrak{h}_1 \rangle}^1(\mathbf{k},V)$. We choose a good splitting $s \in \mathcal{S}pl(\pi^*(c))$ such that $\sigma^{*}(s)$ is the extension  presented by $\tau_{V}$. 

 Let $c(\mathcal{M}):=\mathrm{Loc}(\mathcal{M}, c)\in \Ext_{\mathcal{D}_X}^2(\mathcal{O}_X,\mathcal{V})$. For $ \mathcal{M}_0 \in H^1_{\mathcal{M}}(X, \langle G_1, \mathfrak{h}_1 \rangle)$,  we have $\mathrm{Loc}(\mathcal{M}_0, \pi^*(c))=c(\mathcal{M})$ and 
$\mathrm{Loc}(\mathcal{M}_0, s) \in \mathcal{S}pl(c(\mathcal{M}))$. 
This implies that if $H^1_{\mathcal{M}}(X, \langle G_1, \mathfrak{h}_1 \rangle)$ is non-empty then  $c(\mathcal{M})$ vanishes.
Set
$$\mathrm{Lin}: H^1_{\mathcal{M}}(X, \langle G_1, \mathfrak{h}_1 \rangle ) \longrightarrow    \mathcal{S}pl(c(\mathcal{M})): \mathcal{M}_0 \mapsto \mathrm{Loc}(\mathcal{M}_0, s).$$
The inverse of $\mathrm{Lin}$ is constructed as follows.   
For a splitting $s_X \in \mathcal{S}pl(c(\mathcal{M}))$,  the pull-back $\rho^{*}(s_X)$ is a
splitting of $\rho^{*}c(\mathcal{M})$. 
It has another splitting as follows. Let 
\begin{equation}\label{equ-obstruction}
0 \rightarrow V \rightarrow E_1 \rightarrow  E_2  \rightarrow \mathbf{k} \rightarrow 0
\end{equation}
 be a  Yoneda representation 
of $c$ and  $s$ be  the good splitting of $\pi^{*}(c)$ as above.
Denote by $p$ 
 the projection $\mathcal{M} \rightarrow \mathrm{pt}$.
 The pull-back of \eqref{equ-obstruction} by $p$ is a Yoneda representation of $\rho^{*}c(\mathcal{M})$,  which has the splitting $p^{*}(s)$. The set of isomorphisms between $\rho^{*}(s_X)$ and $p^*(s)$  forms  a groupoid $\mathcal{E}xt^{1}(\mathbf{k},V)=\mathrm{Hom}(\mathbf{k},V)$.
Let $ \mathcal{M}_s$ be the functor that sends each   $ T \in \mathrm{Al}_{\mathrm{k}} $ to the set of pairs $(m,\varphi)$, where $m \in \mathcal{M}(T)$ and  $\varphi \in \mathcal{E}xt^1(\mathbf{k},V)(T)$. 
We can check  that $ \mathcal{M}_s$ forms a $V$-torsor over $\mathcal{M}$ and hence a $\langle G_1, \mathfrak{h}_1 \rangle$-torsor over $X$. The functor given by $$s_X \mapsto  \mathcal{M}_s \in  H^1_{\mathcal{M}}(X, \langle G_1, \mathfrak{h}_1 \rangle )$$  is the inverse $\mathrm{Lin}^{-1}$ of $\mathrm{Lin}$. This shows that if $c(\mathcal{M})$ vanishes then  $H^1_{\mathcal{M}}(X, \langle G_1, \mathfrak{h}_1 \rangle )$ is non-empty. The above discussion can be summarized as follows.   

\begin{prop} \label{thm existence} 
There exists $ c \in H^2(\langle G, \mathfrak{h} \rangle, V) $ such that $ H^1_{\mathcal{M}}(X, \langle G_1, \mathfrak{h}_1 \rangle )$ is non-empty if and only if $\mathrm{Loc}(\mathcal{M}, c) \in H_{\mathrm{dR}}^2(X,\mathcal{V})$ is trivial. In this case  $H^1_{\mathcal{M}}(X, \langle G_1, \mathfrak{h}_1 \rangle )$ is a 
$H_{\mathrm{dR}}^1(X,\mathcal{V})$-torsor.
\end{prop}

 Let $U,V,W$ be a triple of $\langle G, \mathfrak{h} \rangle$-modules such that it  forms an  exact sequence 
\eqref{ses}.  Let  $\langle G_1, \mathfrak{h}_1 \rangle$ be an extension of $\langle G, \mathfrak{h} \rangle$ by $V$ with the class $ c \in H^2(\langle G, \mathfrak{h} \rangle,V)$.  Then  $\langle G_0, \mathfrak{h}_0 \rangle=\langle G_1, \mathfrak{h}_1 \rangle/U$ is an extension of $\langle G, \mathfrak{h} \rangle$ by $W$. We denote by $c_0$ its cohomology class. Assume that  a $\langle G, \mathfrak{h} \rangle$-torsor $\mathcal{M}$ over $X$  can not be lifted  to a $\langle G_1, \mathfrak{h}_1 \rangle$-torsor, but it  can be lifted to a $\langle G_0, \mathfrak{h}_0 \rangle$-torsor.

\begin{prop} \label{thm equivariance}
The map 
$$ H^1_{\mathcal{M}}(X, \langle G_0, \mathfrak{h}_0 \rangle ) \rightarrow H_{\mathrm{dR}}^2(X,\mathcal{U}): \mathcal{M} \mapsto \mathrm{Loc}(\mathcal{M},c_0) $$ is compatible with the  
$H_{\mathrm{dR}}^1(X,\mathcal{W})$-actions on  the two sides.
\end{prop}

\begin{proof}
The $H_{\mathrm{dR}}^1(X,\mathcal{W})$-action on the right-hand side is given by the homomorphism $ d_1: H_{\mathrm{dR}}^1(X,\mathcal{W}) \rightarrow H_{\mathrm{dR}}^2(X,\mathcal{U})$ in the long exact sequence \eqref{les}.
The  following explicit description of 	the $H_{\mathrm{dR}}^1(X,\mathcal{W})$-action on the left-hand side
completes the proof. Suppose that $s \in \mathcal{S}pl(b(c))$ determines  the class $c_0 \in H^{2}(\langle G_0,\mathfrak{h}_0 \rangle, U)$ in the way in \S \ref{sec4.2}.  Now we take $B=\mathbf{k}$ and $A=V$ and consider the long exact sequence \eqref{equ_4.4}. For $ \phi \in \mathrm{Ext}^{1}_{\langle G_0,\mathfrak{h}_0 \rangle}(\mathbf{k},W)$, according to the construction in the last  paragraph of \S \ref{sec4.2}, the twist $\phi \tilde{+}s$ gives us the class $c_0+\tilde{d}_1(\phi)$. Denote by $(\phi \tilde{+} s)_X,(c_0)_X$  and $(\tilde{d}_1(\phi))_X$ the localization of $\phi\tilde{+}s,c_0$  and $\tilde{d}_1(\phi)$ by the torsor $\mathcal{M}$, respectively. By definition of $\mathrm{Lin}^{-1}$, we have 
$$\mathrm{Loc}(\mathrm{Lin}^{-1}((\phi \tilde{+} s)_X),(c_0)_X)=(c_0)_X+(d_1(\phi))_X.$$ 
\end{proof}

\section{Quantization by Harish-Chandra torsors}

Let $ X$ be a superscheme over $\mathbf{k}$  equipped  with  an even Poisson bracket $\{ ,  \}$. 
Recall that $\mathcal{O}_{X,\hbar}$ denotes $\mathcal{O}_{X}[[\hbar]]$ in our convention.
A \textit{deformation quantization} of  $X$ is a sheaf 
$(\mathcal{O}_{X,\hbar},\ast )$ of flat associative  $\mathbf{k}_{\hbar}$-superalgebras such that
\begin{itemize}
	\item[(\rmnum{1})] for any open $U \subset X_{\bbz}$ and $a,b \in \Gamma(U,\mathcal{O}_X)$,  
	$$ [a,b]:= a\ast b-(-1)^{|a||b|}b\ast a \equiv \{a,b\}\hbar \quad (\mathrm{mod} \hbar^{2});$$
	\item[(\rmnum{2})]  the map  $\mathcal{O}_{X,\hbar} \twoheadrightarrow  \mathcal{O}_{X}$  given by specializing $\hbar=0$ is  a  $\mathbf{k}$-superalgebra homomorphism.
\end{itemize}
By a deformation quantization of the symplectic supervariety $(X,\omega)$ we mean that of its associated Poisson supervariety. 
Denote by  $Q(X, \omega)$ the set of   equivalence classes of deformation quantizations up to $\mathbf{k}_{\hbar}$-superalgebra isomorphisms.

 \subsection{ Algebraic objects}
Let $A$ be the complete algebra  $$\mathbf{k}[[x_1,\ldots,x_n, y_1,\ldots,y_n; s_1, \ldots s_m]]$$ 
of formal power series.  It can be seen as the structure sheaf of the formal polydisk $\hat{\mathbf{k}}^{(2n|m)}$ of super dimension 
$(2n|m)$, which is obtained by completing $\mathbf{k}^{(2n|m)}$ at $0$.  Equip $A$ with the standard Poisson bracket
\begin{align}\label{equ_5.1}
\begin{CD}
\{x_i, y_j \}=\delta_{i,j}, &\{s_i, s_j \}=\delta_{i,j},  \\
\{ x_i, s_j\}=\{x_i, x_j\}=&\{ y_i , s_j\}=\{y_i,y_j\}=0.	 
\end{CD}
\end{align}
The super formal Darboux theorem and its quantum version (\cite{SX}) say that  
$A$ is the unique Poisson superalgebra arising from a symplectic form on $\hat{\mathbf{k}}^{2m|n}$ up to  isomorphisms, and $A$ has the quantization $D=A_{\hbar}$  such that  
 \begin{align*}
 [x_i,y_j]=[s_i, & s_j]=\hbar\delta_{i,j},    \\
 [x_i, x_j]=[y_i, y_j]=&[x_i, s_j]=[y_i, s_j]=0,
\end{align*}
which is also unique up to $\mathbf{k}_{\hbar}$-superalgebra isomorphisms. 

We list the main algebraic objects used in this paper.
(\rmnum{1})  $W=\mathrm{Der}A$, the Lie superalgebra of derivations of $A$; 
(\rmnum{2})   $\mathrm{Aut}A$,  the  supergroup of automorphisms of $A$;
(\rmnum{3}) $W_0$, the  vector field preserving the maximal ideal $\mathfrak{m}$ corresponding to the closed point  $0 \in \hat{\mathbf{k}}^{(2n|m)}$; 
(\rmnum{4})  $\mathrm{Sym}A \subset  \mathrm{Aut}A$,  the supergroup of symplectic  automorphisms of $A$;
(\rmnum{5}) $H \subset W$, the Hamiltonian vector fields; (\rmnum{6}) $\mathrm{Aut}D$, the automorphisms of $D$; (\rmnum{7}) $\mathrm{Der}D$ the $\mathbf{k}$-derivations of $D$; (\rmnum{8}) the Lie superalgebra $G=h^{-1}D$; (\rmnum{9}) $(\mathrm{Der}D)_p$, the quotient of $\mathrm{Der}D$ by the ideal consist of  the derivations $\xi \in \mathrm{Der}D$  with  $\xi \cdot D \subset \hbar^{p}D $; (\rmnum{10})   $(\mathrm{Aut}D)_p$, the quotient of $\mathrm{Aut}D$ by the  normal sub-supergroup consist of  the automorphisms $f \in \mathrm{Aut}D$  with  $f \cdot D \subset \hbar^{p}D $; (\rmnum{11}) $G_{p}=G/\hbar^{p+1}G$.
  
As in the non-super case \cite{BK}, we have the following commutative diagram 
\begin{equation}\label{commu-diagram}
	\begin{CD}
		\xymatrix{ & 0 \ar[d] & 0 \ar[d] & 0  \ar[d] & \\
			0 \ar[r] & \hbar^p\cdot\mathbf{k}\ar[r]\ar[d]^{\hbar}  & \hbar^p\cdot A \ar[r]\ar[d] &  \hbar^p\cdot H\ar[r]\ar[d] & 0 \\
			0 \ar[r] & \mathbf{k}[\hbar]/\hbar^{p+2}\ar[r]\ar[d] & G_{p+1}\ar[r]\ar[d] & (\Der D)_{p+1} \ar[r]\ar[d] & 0  \\
			0 \ar[r] &  \mathbf{k}[\hbar]/\hbar^{p+1}\ar[r]\ar[d] & G_{p}\ar[r]\ar[d] &  (\Der D)_{p} \ar[r]\ar[d]  & 0 \\
			& 0  &  0 & 0  & }
	\end{CD}
\end{equation}		
 of homomorphisms of Lie superalgebras.

We also need  a series of Harish-Chandra pairs defined as follows.
Let $V$ be the super vector space spanned by $\{ x_1,\ldots,x_n;y_1,\ldots,y_n \mid s_1, \ldots s_m \}$ and $\mathrm{OSP}(2n|m)$ be the Ortho-symplectic supergroup preserving the super skew-symmetric two form on $V$ given by \eqref{equ_5.1}. For any $T \in \mathrm{Al}_{\mathbf{k}}$, there is a  projection 
$ \phi_P: (\mathrm{Aut} D)_{p}(T) \twoheadrightarrow  \mathrm{OSP}(2n|m) (T)$  determined by
$$ f(x)\equiv \phi(f)(x) \mod (\mathfrak{m}^2(T) +D(T)\hbar)$$
for all $ f \in (\mathrm{Aut} D)_{p}(T)$ and $ x \in V(T)$. We have a semi-direct decomposition 
$(\mathrm{Aut} D)_{p}=\ker(\phi_p) \rtimes \mathrm{OSP}(2m|n)$. 
There is an extension $G_{0,p}$  of the Lie superalgebra $\Lie(\ker(\phi_p))$ by  $\mathbf{k}[\hbar]/\hbar^{p+1}$, which is a nilpotent Lie superalgebra. By [Theorem 3.2, \cite{MO}], there exist  an algebraic supergroup $\tilde{G}_{p}^{0}$ such that $\Lie(\tilde{G}_{0,p})=G_{0,p}$.  Set $\tilde{G}_p=G_{p}^{0}\rtimes \mathrm{OSP}(2n|m)$.  It is clear that $\langle \tilde{G}_p, G_p \rangle $ is a Harish-Chandra pair. We simplify  $\langle \tilde{G}_p,  G_p \rangle$ by $G_p$. We have the following extension of Harish-Chandra pairs
\begin{equation}\label{ext_5.0}
	\begin{tikzcd}[cramped, sep=small]
		1 \arrow[r] &  \mathbf{k}[\hbar]/\hbar^{p+1} \arrow[r] &  G_{p}  \arrow[r] &  \langle (\Aut D)_p , (\Der D)_p \rangle  \arrow[r] &  1.
	\end{tikzcd}	
\end{equation} 
By abuse of  notation, $G$ also denotes the Harish-Chandra pair
obtained by taking $p \rightarrow \infty$. By construction, there is an extension 
\begin{equation}\label{ext_5.2}
	\begin{tikzcd}[cramped, sep=small]
		1 \arrow[r] &  \mathbf{k}_{\hbar}  \arrow[r] &  G  \arrow[r] &  \langle \Aut D , \Der D \rangle  \arrow[r] &  1.
	\end{tikzcd}	
\end{equation}

\subsection{Harish-Chandra torsors}

Let $\mathcal{M}_{\mathrm{coord}}$ denote the \textit{  bundle of formal coordinate systems} on $X$. This is a 
superscheme functor defined as follows.
For each $T \in \mathrm{Al}_{\mathbf{k}}$, $\mathcal{M}_{\mathrm{coord}}(T)$ takes the set of pairs $(\eta, \varphi)$, where $\eta:  \Spec T \rightarrow X $ is a morphism of schemes over $\mathbf{k}$ ( we also write $ \eta \in X(T)$),  
$ \varphi: \hat{\mathcal{O}}_{X,\eta} \rightarrow \mathcal{O}_{\Spec T} \otimes_{\mathbf{k}} A$
is an even isomorphism of $\mathbf{k}$-superalgebras and  $\hat{\mathcal{O}}_{X,\eta}$ is the completion of $\Spec T \times_{\mathbf{k}} X$ at the graph of $\eta$. The superscheme  $\mathcal{M}_{\mathrm{coord}}$ is an  $\Aut(A)$-torsor over $X$ and
the action map $ a:W_{\mathcal{M}_{\mathrm{coord}}} \rightarrow \mathcal{E}_{\mathcal{M}_{\mathrm{coord}}}$ is an isomorphism.
Thus  $\theta_{\mathcal{M}}=a^{-1}$ is a flat connection of  $\mathcal{M}_{\mathrm{coord}}$. Moreover, $\mathcal{M}_{\mathrm{coord}}$ is a transitive $(\Aut A, W)$-torsor over $X$.

\begin{prop}[Smooth super structures]
Let $X_{\bbz}$ be a smooth variety of dimension $k$ over $\mathbf{k}$. Then there exists a one-to-one correspondence between  
the structures of  smooth supervarieties of dimension $(k|m)$ on $X_{\bbz}$ and  $ \langle \Aut A,W \rangle$-torsors over $X_{\bbz}$. Here $A$ and $W$ are defined similarly as above (but $k$ is not necessarily even).
\end{prop}

\begin{proof}
Given a supervariety $X$, we have constructed the  $\langle \Aut A,W \rangle$-torsor $\mathcal{M}_{\mathrm{coord}}$. Its pull-back  $\iota^*\mathcal{M}_{\mathrm{coord}}$ by $\iota: X_{\bbz} \hookrightarrow X$ is an  $ \langle \Aut A,W \rangle$-torsor over $X_{\bbz}$.  Conversely, given an  $\langle \Aut A,W \rangle$-torsor $\mathcal{M}_{\mathrm{coord}}$ over $X_{\bbz}$, 	let $\mathcal{O}_X$ be
 the sheaf of flat sections of  $\mathrm{Loc}(\mathcal{M}_{\mathrm{coord}},A)$. The ringed space  $X=(X_{\bbz},\mathcal{O}_X)$ is a smooth supervariety over $X_{\bbz}$.
\end{proof}
\begin{prop}[Super symplectic structures]\label{sym_tor} 
The set of even super symplectic forms on $X$ is bijective  with the set of 
reductions of the  $\langle \mathrm{Sym} A, H \rangle $-torsors to the $\langle \mathrm{Aut} A, W \rangle$-torsor $\mathcal{M}_{\mathrm{coord}}$;
\end{prop}
\begin{proof}
For any symplectic form $\omega$,  let $\mathcal{M}_{\omega}$ be the scheme functor
which sends each  $T \in \mathrm{Al}_{\mathbf{k}} $ to the set of the pairs $(\eta, \varphi)$, where $ \eta \in X(T)$
and  $\varphi: \hat{\mathcal{O}}_{X,\eta} \simeq \mathcal{O}_{\Spec T} \hat{\otimes}A$
 is an isomorphism of Poisson superalgebras. It is easy to see that $\mathcal{M}_{\omega}$ is
a $\langle \mathrm{Sym} A, H \rangle $-torsor.
Conversely,  given a $\langle \mathrm{Sym} A, H \rangle $-torsor $\mathcal{M}$ over $X$, we need to construct a symplectic form $\omega$ on $X$.  The standard symplectic form $\omega' \in \Omega_{A}^2$ on  $\hat{\mathbf{k}}^{2n|m}$ is $\langle\mathrm{Sym} A, H \rangle$-invariant. The desired symplectic form $\omega$ over $X$  is $\mathrm{Loc}(\mathcal{M},\omega')$.
\end{proof}

For a symplectic form $\omega$ on $X$, the corresponding  $\langle \mathrm{Sym} A, H \rangle $-torsor $\mathcal{M}_{\omega}$ given by Proposition \ref{sym_tor} is  transitive.
It is easy to check that
\begin{align*}
\mathrm{Loc}(\mathcal{M}_{\mathrm{coord}}, A)&=\mathrm{Loc}(\mathcal{M}_{\omega},A)=J^{\infty}(\mathcal{O}_X),   \\
\mathrm{Loc}(\mathcal{M}_{\mathrm{coord}}, \mathbf{k})&=\mathrm{Loc}(\mathcal{M}_{\omega}, \mathbf{k})=\mathcal{O}_X.	
\end{align*}

\subsection{Quantization via super Harish-Chandra torsor}
Recall that the notion $H^1_{\mathcal{M}_{\omega}}(X, \langle \Aut D, \Der D \rangle )$  denotes the set of liftings of $\mathcal{M}_{\omega}$  to an $\langle \Aut D, \Der D \rangle $-torsor with respect to the canonical homomorphism $\langle \Aut D, \Der D \rangle \rightarrow \langle \mathrm{Sym} A, H \rangle$.
\begin{prop}
There is a one-to-one correspondence between $Q(X, \omega)$  and $H^1_{\mathcal{M}_s}(X, \langle \Aut D, \Der D \rangle )$.
\end{prop}	

\begin{proof}
Given  $\mathcal{M}_q \in H^1_{\mathcal{M}_s}(X, \langle \Aut D, \Der D \rangle )$,
we have a quantization of $(X, \omega)$ by taking flat sections of $\mathrm{Loc}(\mathcal{M}_q,D)$. 

Conversely, for given  $\mathcal{D} \in Q(X, \omega)$, we need to construct a $\langle \Aut D, \Der D \rangle $-torsor.  For $T \in \mathrm{Al}_{\mathbf{k}}$ and $\eta \in X(T)$, denote by $p$ the projection $\Spec T \times X \rightarrow X$. 
Let $\mathcal{D}_\eta=p^{*}(\mathcal{D})$, $\mathcal{J}_{\eta} \subset \mathcal{O}_{\Spec T}\times \mathcal{O}_{X}$ be the ideal of graph of $\eta$  and $\widehat{\mathcal{D}_\eta}$ be the completion of  $\mathcal{D_\eta}$ with respect to the ideal $\mathcal{J}_{\eta}+\hbar\mathcal{D}_\eta$.
By the formal quantum Darboux theorem \cite{SX}, $\widehat{\mathcal{D}_\eta}$ is isomorphic to $\mathcal{O}_{\Spec T}\otimes_{\mathbf{k}} D$. 
Let $\mathcal{M}_q$ be the functor sending $T \in \mathrm{Al}_{\mathbf{k}}$ to the set of pairs $(\eta, \Phi)$ where $ \eta \in X(T)$ and $\Phi$ is an  isomorphism between $\mathcal{O}_{\Spec T}\otimes_{\mathbf{k}} D$ and $\widehat{\mathcal{D}_\eta}$.  It is easy to see that  $\mathcal{M}_q$ is a transitive $\langle \Aut D, \Der D \rangle$-torsor and lies in $H^1_{\mathcal{M}_{\omega}}(X, \langle \Aut D, \Der D \rangle )$.
\end{proof}

\subsection{Classification}

\begin{prop}\label{prop5.5}
	Let $(X, \omega)$ be a smooth  symplectic supervariety. 
	If $Q(X,\omega)$ is non-empty  , then it is bijective with $H_{F}^{2}(X)_{\hbar}$.
\end{prop}

\begin{proof} 
	
Denote by  $V$,$U$ and $W$  the $\langle \mathrm{Sym} A , H \rangle$-modules  $A_{\hbar}$, $\mathbf{k}_{\hbar}$ and 
	$H_{\hbar}$, respectively. 
	Let $c$ be the representation class of the extension
	$$\begin{tikzcd}[cramped, sep=small]
		1 \arrow[r] &  W  \arrow[r] &   \langle \Aut D , \Der D \rangle  \arrow[r] &    \langle \mathrm{Sym} A , H \rangle \arrow[r] &  1.
	\end{tikzcd}$$
	Denote by $\mathcal{M}_{\omega}$ be the $\langle \mathrm{Sym} A , H \rangle$-torsor associated with ${\omega}$. Let $ \mathcal{U},\mathcal{V}$ and $\mathcal{W}$ be the localizations of $U,V$ and $W$ by 
	$\mathcal{M}_{\omega}$, respectively. Since a deformation quantization exists, we have that  $\mathrm{Loc}(\mathcal{M}_{\omega}, c)=0$ and $Q(X,\omega)$ is an
	$H^{1}_{\mathrm{dR}}(X,\mathcal{W})$-torsor by 
	Proposition \ref{thm equivariance}.   We claim that there is a  commutative diagram  
	\begin{equation}\label{equ surjec}
		\begin{CD}
			\xymatrix{H_{\mathrm{dR}}^{i-1}(X, \mathcal{W}) \ar[r] \ar@{.>}[d]  &
				H_{\mathrm{dR}}^{i}(X, \mathcal{U})	\ar[r] \ar[d] &  H_{\mathrm{dR}}^{i}(X,\mathcal{V}) \ar[d]\\
				H_{F}^{i}(X)_{\hbar}  \ar[r] & H_{\mathrm{dR}}^{i}(X)_{\hbar} \ar[r] & H^{i}(X,\mathcal{O}_{X})_{\hbar}} 
		\end{CD} 
	\end{equation}
	for all $i\geq 1$. 
	Here the horizontal lines are the long exact sequences in \eqref{les} and \eqref{les3.4}. The second and third vertical lines are the isomorphisms given by  the fact  $\mathrm{Loc}(\mathcal{M}_{\omega}, \mathbf{k})=\mathcal{O}_X$ and
	Corollary \ref{coro_3.7}, respectively.  The squares on the left  of \eqref{equ surjec} are commutative for all $i \geq 0$. The dotted line in \eqref{equ surjec}, which is given in such a way that it makes \eqref{equ surjec} commutative, is an isomorphism. 
\end{proof}

\subsection{Period map}

 Let $c \in   H^2(\langle \Aut , \Der D \rangle, \mathbf{k}_{\hbar}) $ be the  representation class of the extension \eqref{ext_5.2}.
For a quantization $\mathcal{M} \in  H^1_{\mathcal{M}_{\omega}}(X, \langle \Aut D, \Der D \rangle )=Q(X,\omega)$, the period map is defined by
$$\mathrm{Per}:  Q(X,\omega) \rightarrow  H_{\mathrm{dR}}^2(X)_{\hbar};  \mathrm{Per}(\mathcal{M})= \mathrm{Loc}(\mathcal{M},c). $$ 
The period map can be constructed inductively as follows.  Let 
$$c_p \in H^2(\langle (\Aut D)_p , (\Der D)_p \rangle, \mathbf{k}[\hbar]/\hbar^{p+1})$$ 
be the class representing \eqref{ext_5.2}  and set 
$$\mathrm{Per}_p: \langle (\Aut D)_p , (\Der D)_p \rangle  \rightarrow  H_{\mathrm{dR}}^2(X)[\hbar]/\hbar^{p+1};  \mathrm{Per}_p(\bar{\mathcal{M}}_p) = \mathrm{Loc}(\bar{\mathcal{M}}_p,c_p). $$
Here $\bar{\mathcal{M}}_p$ denotes the push-forward of $\mathcal{M}$ by the canonical quotient $ \langle \Aut D, \Der D \rangle  \twoheadrightarrow \langle (\Aut D)_p , (\Der D)_p \rangle$. Then we have $\lim_{p \to \infty}\mathrm{Per}_p =\mathrm{Per}$.

\begin{theorem}\label{thm_peroid} 
Let $(X,\omega)$ be a smooth symplectic supervariety. Assume that $X$ is admissible. 
Then the period  map $\mathrm{Per}$ is injective. Let $F$ be a splitting of the 
natural inclusion $i: H_{F}^{2}(X) \hookrightarrow H_{\mathrm{dR}}^2(X)$, that is, a morphism from 
$H_{\mathrm{dR}}^2(X)$ to $H_{F}^{2}(X)$ satisfying $i\circ F=\mathrm{Id}_{H_{F}^{2}(X)}$.
Denote by  $F_{\hbar}: H_{\mathrm{dR}}^2(X)_{\hbar} \rightarrow H_{F}^{2}(X)_{\hbar}$ the  natural 
extension of $F$.
Then $F_{\hbar}\circ \mathrm{Per}$ is an isomorphism and 
$$\mathrm{Im}(F_{\hbar}\circ\mathrm{Per})=F(\omega)+
\hbar H_{F}^{2}(X)_{\hbar}.$$ 
\end{theorem}	

\begin{proof}
We prove the theorem by an argument  similar to that in \cite{BK}.
Note that $ H^{1}_{\mathcal{M}_\omega} (X, \langle (\mathrm{Der} D)_{0}, (\Aut D)_{0} \rangle)=\{ \mathcal{M}_{\omega}\}$ and  $\mathrm{Per}_{0}(\mathcal{M}_{\omega})=\omega$.
Suppose that the period maps $\mathrm{Per}_{k}$ is  injective  $k \leq p$.  To show the statement for 
$p+1$,  we apply Proposition \ref{thm existence} and \ref{thm equivariance} to the extension 
$$1 \rightarrow V \rightarrow G_{p+1} \rightarrow \langle (\mathrm{Der}D)_{p}, (\Aut D)_{p} \rangle \rightarrow 1,$$ 
and the short exact sequence  
$$0 \rightarrow U  \rightarrow V  \rightarrow W \rightarrow 0 $$ 
 of $ \langle \mathrm{Der}(D)_{p}, \Aut(D)_{p} \rangle $-modules. 
Here $U=\mathbf{k}[\hbar]/\hbar^{p+2}$, $V=\mathbf{k}[\hbar]/\hbar^{p+2} \oplus A\hbar^{p}$ and  $W=A\hbar^{p}/\mathbf{k}\hbar^{p}=H\hbar^{p}$. Let $\mathcal{U}, \mathcal{V}$ and $\mathcal{W}$ 
be the corresponding localizations by $\mathcal{M}_p$, respectively.  Recall the maps $b_2$ and $d_1$  in the long exact sequence \ref{les}. As in the non-super case \cite{BK}, 
under the admissibility assumption, 
$b_2$ is trivial and  $d_1$ is injective. This follows from a commutative diagram  constructed as \eqref{equ surjec}. 
Thereby  $\mathrm{Loc}(\mathcal{M}_p, b_2(c))=0$. Theorem \ref{thm existence} implies that $H^1_{\mathcal{M}_p}(X, \langle (\Aut D)_{p+1}, (\Der D)_{p+1} \rangle )$ is non-empty.
By Theorem \ref{thm equivariance}, we have 
\begin{equation}
 \mathrm{Per}_{p+1}: H^1_{\mathcal{M}_p}(X, \langle (\Aut D)_{p+1}, (\Der D)_{p+1} \rangle ) \rightarrow H_{\mathrm{dR}}^{2}(X, \mathcal{U})=H_{\mathrm{dR}}^{2}(X)[\hbar]/\hbar^{p+2}
\end{equation}
that is $H_{\mathrm{dR}}^{1}(X, \mathcal{W})$-equivariant.
The  left-hand side is a $H_{\mathrm{dR}}^{1}(X, \mathcal{W})$-torsor. Since $d_1$ is injective, the  $H_{\mathrm{dR}}^{1}(X,\mathcal{W})$-action on the right-hand side is free. Thus we show that $\mathrm{Per}_{p+1}$ is injective. 

The last statement follows from the facts that $\mathrm{Per}$ is 
$H_{\mathrm{dR}}^{1}(X, \mathcal{W})_\hbar=H_{F}^{2}(X)_\hbar$-equivariant and   
$ H_{F}^{2}(X)_\hbar \rightarrow H_{\mathrm{dR}}^{2}(X)_\hbar$ is injective.
\end{proof}

Recall that  $\iota: X_{\bbz} \hookrightarrow X $ is the canonical embedding. We have a symplectic variety $(X_{\bbz},\omega_{\bbz})$ with $\omega_{\bbz}:=\iota^{*}(\omega)$. 
It is not easy to relate quantizations of  $(X, \omega)$ and  $(X_{\bbz}, \omega_{\bbz})$  directly. 
However, we can relate them by the period map as follows. 

By Proposition \ref{prop_admi} (\rmnum{2}), there exist splittings $F: H_{\mathrm{dR}}^2(X) \twoheadrightarrow H_{F}^2(X)$ and  $F_{\bbz}: H_{\mathrm{dR}}^2(X_{\bbz}) \twoheadrightarrow H_{F}^2(X_{\bbz})$ to the maps in the first column of \eqref{equ-surjec}, respectively, such that the following diagram 
$$\begin{tikzcd}
& H_{\mathrm{dR}}^2(X) \arrow[r, "F"] \arrow[d] & H_{F}^2(X) \arrow[d] \\
& H_{\mathrm{dR}}^2(X_{\bbz})  \arrow[r, "F_{\bbz}"] & H_{F}^2(X_{\bbz})
\end{tikzcd}$$
is commutative.
\begin{theorem}\label{thm_5.4}
Assume that the symplectic supervariety $(X,\omega)$ and its reduced part $(X_{\bbz},\omega_{\bbz})$ are admissible. Fix two splittings $F$ and $F_{\bbz}$ as above. Let $F_{\hbar}$ and $(F_{\bbz})_{\hbar}$ defined similarly as in Theorem \ref{thm_peroid}. Then there exists a unique injective map $\iota_Q$ such that the diagram
	\begin{equation} \label{dia_5.7}
	 \begin{CD}
	 	\xymatrix{
		Q(X, \omega) \ar[rr]^{F_{\hbar}\circ\mathrm{Per}}\ar[d]_{\iota_Q} & & H_{F}^2(X)_{\hbar} \ar[d] \\
		Q(X_{\bbz}, \omega_{\bbz}) \ar[rr]^{(F_{\bbz})_\hbar\circ\mathrm{Per}_{\bbz}}& &  H_{F}^2(X_{\bbz})_{\hbar}}
	\end{CD}
	\end{equation}
is commutative.  
 \end{theorem}

\begin{proof}
 The right vertical arrow is injective by  Proposition \ref{prop_admi} (\rmnum{2}) and 
 horizontal arrows are isomorphisms by Theorem \ref{thm_peroid}.
\end{proof}
\begin{remark}\label{rmk_5.7}
 Assume that the supervariety $(X,\omega)$ and its reduced variety $(X_{\bbz},\omega_{\bbz})$ are admissible.  An interesting question is whether we have  $\mathrm{Im}(\mathrm{Per}) \subset \mathrm{Im}(\mathrm{Per}_{\bbz})$. If this holds, then the right vertical arrow and horizontal arrows of \eqref{dia_5.7}
 can be replaced by the middle vertical arrow in \eqref{equ-surjec} and by the period maps, respectively.
\end{remark}

\section{Deformation quantizations of super coadjoint orbits}
 This section studies  deformation quantizations of certain super nilpotent orbits from the basic Lie superalgebras over $\mathbf{k}$.  We start by checking the admissibility and splitness in a more general setting.  

\subsection{ Split supervarieties}
Let $X=(X_{\bbz},\mathcal{O}_{X})$ be a smooth supervariety and $Y_{\bbz}$ be an affine, irreducible and  Cohen-Macaulay variety. Assume that there is an open inclusion $i: X_{\bbz} \hookrightarrow Y_{\bbz}$ such that $i(X_{\bbz})$ coincides with the smooth locus of $Y_{\bbz}$  and $\mathrm{codim}_{Y_{\bbz}}(Y_{\bbz}/X_{\bbz}) \geq 2$.

\begin{lemma}[\cite{KP2}, Lemma 9.1]\label{prop-exten} 
Assume that the varieties $X_{\bbz}$ and $Y_{\bbz}$ 
satisfy the above conditions.
For any affine open subset  $U \subset Y_{\bbz}$, 	
$$\Gamma(U,i_{*}\mathcal{O}_{X_{\bbz}})=\Gamma(U ,\mathcal{O}_{Y_{\bbz}}).$$	
\end{lemma}

 Lemma \ref{prop-exten} implies  that $ i_{*}(\gr\mathcal{O}_{X})$ is a locally free sheaf over $Y_{\bbz}$. We regard   $(Y_{\bbz}, i_{*}(\gr\mathcal{O}_{X})) $ as a singular supervariety. 
 Let $E$ be the vector bundle over $Y_{\bbz}$ such that $ \bigwedge^{\bullet} E=i_{*}(\gr\mathcal{O}_{X})$, and $\mathcal{K}$ be the sheaf of unipotent  groups defined by 
 $$ \mathcal{K}(U)=\{g \in \Aut(\bigwedge^{\bullet} E)(U)| (g- \mathrm{id}) (\bigwedge^{i} E(U)) \subset  \bigwedge^{i+2} E(U) \}.$$
 
 If $Y_{\bbz}$ is smooth,
 the supervarieties $Y$ over $Y_{\bbz}$ with $\gr(Y)=(Y_{\bbz},\bigwedge^{\bullet} E)$ are classified by   $H^{1}(Y_{\bbz},\mathcal{K})$ up to isomorphisms \cite{Pe1}.  By Lemma \ref{prop-exten},  this statement can be extended to the singular supervarieties $Y_{\bbz}$ which satisfy the conditions therein.  Let $\mathfrak{k}$ be the sheaf of  vector fields given by 
$$\mathfrak{k}(U)=\Lie(\mathcal{K}(U)).$$
Since the group $\mathcal{K}(U)$ is unipotent, it is bijective with 
its Lie algebra  $\Lie(\mathcal{K}(U))$ as sets. Hence $H^{1}(Y_{\bbz},\mathcal{K})=H^{1}(Y,\mathfrak{k})$ as sets.
Since $Y_{\bbz}$ is affine, $H^{1}(Y_{\bbz},\mathfrak{k})=0$.
By the same argument in the proof of [Corollary 3.4 \cite{Pe1}], we can show that
$(Y_{\bbz},i_{*}\mathcal{O}_{X})$ is split. Restricting to the smooth locus, we show that $X$ is split. The above discussion can be summarized as follows.   
\begin{prop}\label{prop_split}
Assume that $X$, $X_{\bbz}$ and $Y_{\bbz}$ 
satisfy the above conditions. 
Then the ringed space $(Y_{\bbz},i_{*}\mathcal{O}_{X})$ is,  and accordingly, $X$ is a split supervariety.  
\end{prop}

\subsection{Admissible supervarieties}
We recall the notion of depth and its relation with the cohomology of sheaves with supports in a closed subvariety.  Let $A_{\bbz}$ be a commutative $\mathbf{k}$-algebra, $\mathfrak{a}$ be an ideal of $A_{\bbz}$ and $M$ be an $A_{\bbz}$-module.  We say that a sequence $x_1, \ldots, x_{n} \in A_{\bbz} $ is $M$-\textit{regular} if the action map $ m \mapsto  x_im$ is an injective $A_{\bbz}$-module 
autohomomorphism  of $M/(x_1,\ldots,x_{i-1})M$  for all $i$, $ 1 \leq i \leq n$. The \textit{depth} of $M$ over $\mathfrak{a}$, which is denoted by $\mathrm{depth}_\mathfrak{a}M$, is defined as the maximum length of a $M$-regular sequence $x_1, \ldots, x_{n}$ 
with  all $x_i \in \mathfrak{a}$. In the case of $\mathfrak{a}=A_{\bbz}$, $\mathrm{depth}_\mathfrak{a}M$ is called the depth of $M$ and will be simplified by 
$\mathrm{depth}M$.

Let $A_{\bbz}=\mathbf{k}[Y_{\bbz}]$, $Z_{\bbz}=Y_{\bbz}/X_{\bbz}$ be the singular locus of $Y_{\bbz}$ and $\mathfrak{a}\subset A_{\bbz}$  the ideal of functions vanishing on $Z_{\bbz}$. For a sheaf $\mathcal{F}$  on $Y_{\bbz}$, denote by $H^{r}_{Z_{\bbz}}(Y_{\bbz},\mathcal{F})$ the cohomology of $\mathcal{F}$ with supports in $Y_{\bbz}$.
The following lemma is given in [\cite{Ha},Exercise \Rmnum{3},2.3, 3.3 and 3.4]. 
\begin{lemma}\label{depth_Lemma}
\begin{itemize}
	\item[(\rmnum{1})] We have the  long exact sequence  
	$$\cdots \rightarrow H^{r}_{Z_{\bbz}}(Y_{\bbz},\mathcal{F}) \rightarrow H^{r}(Y_{\bbz},\mathcal{F}) \rightarrow  H^{r}(X_{\bbz},\mathcal{F}|_{X_{\bbz}}) \rightarrow \cdots $$
	of sheaf cohomology.
	\item[(\rmnum{2})] For an $A_{\bbz}$-module 
	$M$, we have that $\mathrm{depth}_\mathfrak{a} M \geq r $ if and only if $H^{i}_{Z_{\bbz}}(\tilde{M})=0$ for all $i<r$. Here $\tilde{M}$ denotes the sheaf on $Y_{\bbz}$ associated to $M$.
\end{itemize}	
\end{lemma}

\begin{prop}\label{prop6.4}
Let $X_{\bbz}$ be the smooth locus of an affine, irreducible and Cohen-Macaulay variety $Y_{\bbz}$. Assume that $\mathrm{codim}_{Y_{\bbz}}(Z_{\bbz})\geq 2$ and $H^1(X_{\bbz},\mathcal{O}_{X_{\bbz}})=H^2(X_{\bbz},\mathcal{O}_{X_{\bbz }})=0$. Then for any smooth supervariety $X=(X_{\bbz},\mathcal{O}_X)$ on $X_{\bbz}$, we have $H^{1}(X,\mathcal{O}_X)= H^{2}(X,\mathcal{O}_X)=0$.
\end{prop}

\begin{proof}
Since $Y$ is affine, $H^{3}(Y_{\bbz},\mathcal{O}_{Y_{\bbz}})=0$. 	
 It follows from the assumption   $H^1(X_{\bbz},\mathcal{O}_{X_{\bbz}})=H^2(X_{\bbz},\mathcal{O}_{X_{\bbz}})=0$
and Lemma \ref{depth_Lemma} (\rmnum{1}) 
 that $H^{i}_{Z_{\bbz}}(Y_{\bbz},\mathcal{O}_{Y_{\bbz}})=0$ for $ 0 \leq r \leq 3$. Hence, $\mathrm{depth}_{\mathfrak{a}}(A_{\bbz})\geq 4$ by Lemma \ref{depth_Lemma} (\rmnum{2}). Thus there exist an $A_{\bbz}$-regular 
sequence $x_1,x_2, x_3, x_4 \in \mathfrak{a}$. Let $M=\Gamma (Y_{\bbz},i_{*}\mathcal{O}_{X})$. The singular supervariety $(Y_{\bbz},i_{*}\mathcal{O}_{X})$ is split by Proposition \ref{prop_split}. 
 Hence $M$ is an $A_{\bbz}$-module with $\tilde{M}=i_{*}\mathcal{O}_{X}$. We claim that image of $x_1,x_2, x_3, x_4 \in \mathfrak{a} $ under the inclusion 
$ A_{\bbz} \hookrightarrow M $ given by the splitness is a $M$-regular sequence. Indeed, suppose that 
$x_{i}m=0$ for some $m \in M/(x_1,\ldots,x_{i-1})M$. Since $\tilde{M}$ is 
locally free, for any $x \in Y_{\bbz}$, there exists an open subset  $V \subset Y_{\bbz}$ such that $x \in V$ and $\tilde{M}|_V$ is a  free $\mathcal{O}_{V}$-module.  Since $Y_{\bbz}$ is irreducible, $x_1,\ldots,x_{4}$ is a $\Gamma(\tilde{M},V)$-regular sequence. Thus $x_{i}m|_{V}=0$ implies  $m|_{V}=0$.
Hence $m=0$ and the claim follows. Lemma \ref{depth_Lemma} (\rmnum{2}) implies that $H^{i}_{Z_{\bbz}}(Y_{\bbz},\iota _*\mathcal{O}_{X})=0$ for $ 0 \leq r \leq 3$. The proposition follows from  Lemma \ref{depth_Lemma}. 
 \end{proof}

\subsection{ Quantizations of super nilpotent orbits}

Let $G$ be an algebraic supergroup and assume that $\mathfrak{g}=\Lie (G)$ is a basic Lie superalgebra.  Given  $T \in \mathrm{Al}_{\mathbf{k}}$, there is a natural adjoint action of the group $G(T)$ on $\mathfrak{g}(T)$. This defines the adjoint action of $G$ on $\mathfrak{g}$ as group superschemes. For an even element $e 
\in \mathfrak{g}_{\bar{0}}$, the adjoint orbit $\mathbb{O}$ of $e$ is the superscheme 
given by $$\mathbb{O}(T)=G(T)\cdot e(T)$$ for $T \in \mathrm{Al}_{\mathbf{k}}$. The even reduced scheme $ \mathbb{O}_{\bbz}$ of $\mathbb{O}$ is the ordinary  adjoint orbit $G_{\bbz}e$. 
Since $\mathfrak{g}$ is a basic Lie superalgebra,
there is an even, non-degenerate, $G$-invariant two form $(,)$ on  $\mathfrak{g}$. Thus we have an  isomorphism $\mathfrak{g} \rightarrow  \mathfrak{g}^{*}$ which sends $e$ to $\chi:=(e,\text{-})$.
Furthermore,  $\mathbb{O}=Ge$ is isomorphic to  the coadjoint orbit $G\chi$. According to \cite{Ko},  there is a  Kostant-Kirillov 
symplectic form $\omega_{\chi}$ on $\mathbb{O}$ given by 
$$ \omega_{\chi}(\bar{x},\bar{y})(a)=a([x,y])$$ 
for  $\bar{x},\bar{y} \in  (\mathrm{T}_{\mathbb{O}})_a=\mathfrak{g}/\mathfrak{g}_{a}$ and $a \in (G\chi)_{\bbz}$. Here $\bar{x},\bar{y}$ are the equivalence classes of $x,y \in \mathfrak{g}$, respectively. 
\begin{theorem}
Let $(\mathbb{O}, \omega_{\chi})$ be the adjoint orbit of a nilpotent $e \in \mathfrak{g}_{\bar{0}}$, equipped with the Kostant-Krillov form $\omega_{\chi}$. Assume  that the closure of the even reduced part $\mathbb{O}_{\bbz}$ is irreducible, Cohen-Maucalay and 
$H^1(\mathbb{O}_{\bbz},\mathcal{O}_{\mathbb{O}_{\bbz}})=H^2(\mathbb{O}_{\bbz},\mathcal{O}_{\mathbb{O}_{\bbz }})=0$. Then we have that $\mathbb{O}$ is split and  $H^1(\mathbb{O},\mathcal{O}_{\mathbb{O}})=H^2(\mathbb{O},\mathcal{O}_{\mathbb{O}})=0$;
 the set $Q(\mathbb{O},\omega_{\chi})$ of equivalence classes of deformation quantizations of $(\mathbb{O}, \omega_{\chi})$  is bijective with $ H_{\mathrm{dR}}^{2}(\mathbb{O})_{\hbar}= H_{\mathrm{dR}}^{2}(\mathbb{O}_{\bbz})_{\hbar}$.	
\end{theorem}
   
\begin{proof}
Let $X=\mathbb{O}$, $X_{\bbz}=\mathbb{O}_{\bbz}$ and  $Y_{\bbz}=\bar{\mathbb{O}}_{\bbz} \subset \mathfrak{g}_{\bar{0}}$ be the Zariski closure of $X_{\bbz}$. The first statement follows from Proposition \ref{prop_split} and \ref{prop6.4}. 
By Proposition \ref{prop_admi} (\rmnum{2}) we have 
$$ H_{F}^{2}(X_{\bbz})=H_{\mathrm{dR}}^{2}(X_{\bbz})=H_{\mathrm{dR}}^{2}(X)=H_{F}^{2}(X).$$
The last statement follows from Proposition  \ref{prop5.5} and Theorem \ref{thm_peroid}.
\end{proof}	

\begin{remark}
There are close relations between normality and Cohen-Macaulaness of the closures of nilpotent orbits \cite{KP1, KP2}. A list for the nilpotent orbits $\mathbb{O}_{\bbz}$ with $H^1(\mathbb{O}_{\bbz},\mathcal{O}_{\mathbb{O}_{\bbz}})=H^2(\mathbb{O}_{\bbz},\mathcal{O}_{\mathbb{O}_{\bbz}})=0$ is given in \cite{Lo1}. 
Some basic properties of super nilpotent orbits are given in \cite{Mu2}.
\end{remark}
As in the non-super case \cite{Lo2}, quantization of  
general nilpotent orbits can be studied by that of super Q-factorial terminalizations. This will appear somewhere else.

\subsection*{Data availability}
Data  sharing is not  applicable to this article as no new data were created or analyzed in
this study.

\section*{Declarations}

\subsection*{Conflict of interest}
The author declares no conflict of interest.

\end{document}